\documentclass[draft,12pt,leqno,twoside]{amsart}
 \usepackage{amssymb,amscd}
\usepackage{epsf, amsmath,amsbsy,amsfonts, amsthm, amssymb,upref,yfonts}
 \usepackage{amsrefs}
 \usepackage{stmaryrd}
\usepackage{tikz}

\usepackage{tikz-cd}
\usetikzlibrary{decorations.pathmorphing}

\usepackage{tkz-euclide}
\tikzset{snake it/.style={decorate, decoration=snake}}
\usetikzlibrary{shapes.geometric,positioning,decorations.pathreplacing}

\keywords{Stochastic embeddings, Distortion, Slashpowers, Trees}
\subjclass[2010]{05C5, 68R10}

\newtheorem{thm}{Theorem}[section]
\newtheorem{lem}[thm]{Lemma}
\newtheorem{cor}[thm]{Corollary}
\newtheorem{prop}[thm]{Proposition}
\theoremstyle{definition}

\newtheorem{exs}[thm]{Examples}

\newtheorem*{rem}{Remark}

\newtheorem*{thmmain}{Main Theorem}

\usepackage{amssymb}

\newcommand{\N}{\mathbb{N}}

\renewcommand{\P}{\mathbb{P}}
\newcommand{\R}{\mathbb{R}}

\newcommand{\D}{\mathbb{D}}

\newcommand{\E}{\mathbb{E}}

\newcommand{\cC}{\mathcal C}

\newcommand{\cM}{\mathcal M}

\newcommand{\cS}{\mathcal S}
\newcommand{\cT}{\mathcal T}

\newcommand{\keq}{\!=\!}

\newcommand{\ie}{{\it i.e.,}} 

\newcommand{\length}{\text{\rm length}}

\title{Stochastic Embeddings of Graphs  into Trees}
\author{Th. Schlumprecht}
\address{Department of Mathematics, Texas A\&M University, College Station, TX 77843}
\email{t-schlumprecht@tamu.edu}

\author{G. Tresch}
\address{Department of Mathematics, Texas A\&M University, College Station, TX 77843}
\email{treschgd@tamu.edu}

\thanks{The research was partially supported by the National Science Foundation under Grant Number DMS-2054443. This paper is part of the Ph.D. thesis
of the second named author.}

\begin{document}
\begin{abstract}  It is known that every graph with $n$ vertices embeds stochastically into trees with distortion $O(\log n)$. In this paper, we 
 show that this upper bound is sharp for a large class of graphs.  As this class of graphs contains diamond graphs, this result extends known examples that obtain this largest possible stochastic distortion.
\end{abstract}

\maketitle

\tableofcontents
\allowdisplaybreaks

\section{Introduction}

The continued development of the theory of low distortion embeddings of general finite metric spaces into simple metric spaces has been important for understanding how complex metric structures can be efficiently represented. Furthermore, the development of such theory often results in fruitful applications for computer science approximation algorithms (see \cite{BhattLeighton1984,LeightonRao1999,KleinRaoAgrawalRavi1995} for classic applications involving the sparsest cut problem and how it relates to embeddings into $\ell_1$ metrics). Due to this algorithmic connection, a particularly desirable family of simple metric spaces is weighted trees. However, there are simple graphs  that cannot be well approximated by trees. For example, any embedding of the $n$-cycle into a tree must incur a distortion of at least order $n$ \cite{RabinovichRaz1998}
 (see also Lemma \ref{L:3.1} below). Nevertheless, there has been progress when the problem is generalized to stochastic embeddings.

Stochastic embeddings (also called random or probabilistic embeddings) were originally defined by Bartal \cite{Bartal1996}. Since finite, weighted trees are isometrically isomorphic to finite dimensional $\ell_1$ spaces, the existence of a stochastic embedding of a metric space into a family of trees implies the existence of a bi-Lipshchitz embedding into $\ell_1$ without an increase of the distortion. More recently, it has been shown that a stochastic embedding of a finite metric space into a family of trees can be used to obtain bi-Lipshitz embeddings of lamplighter metric spaces and transportation cost spaces into $L_1$ (see \cite{BMSZ2022,DilworthKutzarovaOstovskii2021,Mathey-PrevotValette2021}).

Also in \cite{Bartal1996} Bartal shows that any finite metric space on $n$ elements can be stochastically embedded into trees with stochastic distortion $O(\log^2(n))$. This upper bound was lowered to $O(\log(n))$ by Fakcharoenphol et al \cite{FakcharoenpholRaoTalwar2004} which was known to be optimal by the  observation of Bartal that all expander graphs must exhibit $\Omega(\log(n))$ stochastic distortion \cite{Bartal1996}. More basic families of metric spaces that achieve this bound have been since found such as the 2-dimensional grid \cite{AlonKarpPelegWest1995} and a family of graphs, called the diamond graphs \cite{GuptaNewmanRabinovichSinclair1999}. Our main result is a generalization of the latter result.

 \begin{thmmain} Suppose $(G,d_{G})$ is a normalized geodesic  $s$-$t$ graph. For $k\in\N$  let $(G^{\oslash k}, d_{G^{\oslash k}})$ be its $k$-th slash power, and $\mathcal{S}_{\mathcal{T}}(G^{\oslash k}, d_{G^{\oslash k}})\ge1$ be the stochastic distortion 
 of $(G^{\oslash k}, d_{G^{\oslash k}})$ into the family of geodesic trees $\cT$.
 Then
\[\mathcal{S}_{\mathcal{T}}(G^{\oslash k}, d_{G^{\oslash k}})= \begin{cases}
    \Theta(k) & G\ \text{contains a cycle,} \\
    1 & G\ \text{is a path.}
\end{cases}\]

\end{thmmain}

Thus, slash powers provide a large family of graphs with the largest possible stochastic distortion.

The paper is organized as follows. In Section 2, we state the necessary notation and relevant definitions, including the construction of slash powers of graphs and \textit{expected distortion}. In addition, we show that given a probability measure on the edges of a graph the corresponding expected distortion of a map into an expanding tree is a lower bound for the stochastic distortion of the graph into a family of trees. In Section 3, we state a mild strengthening of a result of Gupta \cite{Gupta2001} on the lower bounds for distortion of a weighted cycle into a tree. Then in Section 4, we prove the desired result for slash powers of a vital family of weighted graphs that we call \textit{balanced generalized Laakso graphs}. A proof of the Main Theorem is then given in Section 5. \\

\section{Main Definitions and Preliminary Results}\label{S:2}

\medskip\noindent 2.1. {\bf Graphs}\label{SS:2.1}
For a simple, finite graph $G=(V,E)$ we let $V(G)$ be the vertex set $V$ and $E(G)\subset[V(G)]^2=\big\{\{u,v\}: u,v\in V(G), v\not=u\big\}$ be the set of edges $E$. In the case that $G$ is a directed graph, we use $E(G)$ to denote the edges of $G$ without orientation and $E_d(G)$ to denote the directed edges under an orientation $d$.
 For $e=(u,v)\in E_d(G)$ we put $e^-=u$ and $e^+=v$.

A {\em subgraph of $G$ } is  a graph $G'$, where $V(G')\subset V(G)$ and 
$E(G')\subseteq E(G)\cap\big\{\{u,v\}: u,v\in V(G')\big\}$. 

We call a subgraph $P=(V(P),E(P))\subset (V(G),E(G))$ \textit{a path in }$G$ if there is $k\in \N$ and a labeling of the distinct vertices of $P$, $V(P)=\{x_0,x_1,\ldots, x_k\}$ such that 
$E(P)=\{\{x_{i-1},x_i\} : i\in \{1,\ldots, k\}\}$. We call the above subgraph $P$ \textit{a directed path in } $G=(V(G), E_d(G))$ if, in addition, the edges of $P$ agree with the orientation of $E(G)$ (i.e. $E_d(P)=\{(x_{i-1},x_i) : i\in \{1,\ldots, k\}\}$ where $d$ is a fixed orientation on $E(G)$). We shall often denote a (directed or undirected) path by $ P=(x_i)_{i=0}^k$ or $(x_0,\ldots, x_k)$
and call it a {\em path from $x_0$ to $x_k$}, or {\em between $x_0$ and $x_k$}.
For two paths $A:=(a_i)_{i=0}^{p}$ and $B:=(b_i)_{i=0}^q$ with $a_p=b_0$, we use $A\smile B$ to denote the concatenation of the two paths at the vertex $a_p$. More specifically, $V(A\smile B)=\{a_0,a_1,\ldots, a_p=b_0,b_1,\ldots, b_q\}$ and $E(A\smile B)=\{\{a_{i-1},a_i\}: i\in \{1,\ldots, p\}\}\cup\{\{b_{i-1},b_i\}: i\in \{1,\ldots, q\}\}$. Note that the concatenation of the two paths above is a path iff $V(A)\cap V(B)=\{a_p=b_0\}$.
 
 A subgraph $C=(V(C),E(C)) \subset (V(G),E(G))$ \textit{a cycle in }$G$, if   the distinct vertices of $C$, can be ordered into  $V(C)=\{x_0,x_1,\ldots, x_k\}$, with $k\ge 3$, and $x_0=x_k$, such that 
$E(P)=\big\{\{x_{i-1},x_i\} : i\in \{1,\ldots, k\}\big\}$.

A graph  $G$ is called {\em connected}, if for any two vertices there is a path between them.
 A {\em tree } is a connected graph that does not contain a cycle, or equivalently, between any two vertices there is a unique path between them.  
 
\medskip\noindent 2.2. {\bf Geodesic Metrics}\label{SS:2.2}
If $G$ is a connected graph and  $d_G$ is a metric on $V(G)$, we call $d_G$ a {\em geodesic metric on $G$} if 
$$d_G(u,v)=  \min\{ \length_{d_G}(P): P\text{ is a path from $u$ to $v$}\},\text{ for $u,v\in V$,}$$
where for a path $P=(x_j)_{j=0}^n$ in $G$, we define the length of $P$ by
$$\length_{d_G}(P)=\sum_{j=1}^n d_G(x_{j-1},x_j).$$
and sometimes refer to this value as the \textit{metric length of} $P$. In this case, we call the pair $(G,d_G)$ a {\em geodesic graph}.
For $e=\{u,v\}\in E(G)$ we denote $d_G(e)=d_G(u,v)$.

Assume that $w:E(G)\to \R^+$ is a function. 
Define for  $u,v\in V(G)$
$$d_G(u,v):=\min\Big\{ \sum_{j=1}^n w(\{x_{j-1},x_j\}) : (x_j)_{j=0}^n \text{ is a path from $u$ to $v$}\Big\}.$$
Then $d_G$ is a geodesic metric on $G$, and we call it the {\em metric generated by the weight function} $w$.
Note that if $w:E(G)\to \R^+$ is an arbitrary function and $d_G$ the geodesic metric generated by $w$, it does not necessarily
follow that for an edge $e$ we have $d_G(e)=w(e)$, since it might be possible that there is a  path of shorter metric length between the two endpoints of $e$.

For the special weight function, defined by $w(e)=1$ for all $e\in E(G)$, and for a path $P$ in $G$ we call $\sum_{e\in E(P)}w(e)=|E(P)|$ the {\em graph length of $P$}.

Conversely, any geodesic metric on $G$ is generated by the weight function
$$w: E\to \R^+, \qquad \{u,v\}\mapsto d_G(u,v).$$

If $(G,d_G)$ is a geodesic graph and $H=(V(H),E(H))$ is a connected subgraph
we call $(H,d_H)$, where $d_H$  is the geodesic distance on $V(H)$ generated by the weight
$w: E(H)\to \R^+$, $e\mapsto d_G(e)$, an \textit{induced geodesic subgraph} and $d_H$ the \textit{induced geodesic metric of $d_G$ on $V(H)$}. Note that $d_H$ is not necessarily the restriction of $d_G$ on $V(G)$ (for example, if $G$ is a cycle and $H$ is obtained
by taking away one edge). In the scenario when $d_H(x,y)=d_G(x,y)$ for all $x,y\in V(H)$ we say that $H$ is an \textit{isometric geodesic subgraph} of $G$ or, when the context is clear, simply an \textit{isometric subgraph}.\\

\medskip\noindent{2.3. \bf $s$-$t$ Graphs}\label{SS:2.3}
We call a connected, graph $G=(V(G),E(G))$ an {\em $s$-$t$ graph} if it has  two distinguished vertices denoted by  $s=s(G)$ and $t=t(G)\in V(G)$
so that $G$ can be turned into a directed graph $G_d=(V(G), E_d(G))$ where every edge  $e\in E_d(G)$ is an element of a  directed  path from $s(G)$ to $t(G)$. When given such an orientation, we call $G$ a \textit{directed} $s$-$t$ graph and note that under this definition every vertex $v\in V(G)$ lies on a directed path from $s(G)$ to $t(G)$.

Let $G=(V(G),E(G))$ be an $s$-$t$ graph and let $d_G$  be a geodesic metric on $G$.
We say that $(G, d_G)$ is a {\em geodesic $s$-$t$ graph}, if all paths from $s(G)$ to $t(G)$ have the same length, and we  say that 
 $(G, d_G)$ is a  {\em normalized geodesic $s$-$t$ graph}, if that length is $1$. This property has the following important, and in this paper often used, consequence.
 \begin{prop}\label{P:2.3.0} If $(G, d_G)$ is a geodesic $s$-$t$ graph, every path $P=(x_j)_{j=0}^l$ from $x_0=s(G)$ to $x_l=t(G)$ is an isometric subgraph of $G$.
 In particular, this means that 
 $$d_G(x_i,x_j)=\sum_{s=i+1}^j d_G(x_{s-1},x_s) \text{ for $0\le i<j\le l$}.$$
 \end{prop}

  In the case where $d_G(e)=1$ for all $e\in E(G)$ for a geodesic $s$-$t$ graph $(G,d_G)$ then $(G,d_G)$ is sometimes referred to as a \textit{bundle graph}\cite{GuptaNewmanRabinovichSinclair1999}.\\

\begin{exs}\label{Ex:2.3.1} We list three elementary examples of   geodesic $s$-$t$ graphs:
\begin{enumerate}
\item[a)] ({\em Paths}):  Let
 $P=(x_i)_{i=0}^k$ be a path. If we define $s(P)=x_0$ and $t(P)=x_k$ then it is clear that $P$ is an $s$-$t$ graph. Let $w:E(P)\to\R^+$ be a weight function, then $w$ generates a geodesic metric $d_P$ on $P$ making $(P,d_P)$ a  geodesic $s$-$t$ graph.
 
 \item[b)] ({\em Cycles}):  Let $C=(V(C),E(C))$ be a cycle, with  $V(C)=\{x_0,x_1, \ldots, x_{l-1}\}$, 
 with $l\ge 3$, such that $E(C)=\big\{ \{x_{j-1}, x_j\}: j=1,2,\ldots, l\big\}$, where $x_l=x_0$. Let 
 $d_C$ be a geodesic metric on $V(C)$,  generated by  a weight function $w:E(C)\to \R^+$ and assume that there
 is an $1\le m\le l-1$ 
 so that $$d_C(x_0,x_m)=\sum_{j=1}^m w(\{x_{j-1}, x_j\}) =\sum_{j=m+1}^{l}  w(\{x_{j-1}, x_{m-j}\}).$$
 We can then  orient $E(C)$ by
 $$E_d(C)=\{(x_{i-1},x_i): i\in \{1,\ldots, m\}\}\cup \{(x_{j},x_{j-1}):\ j\in \{m+1,\ldots, l\}\}.$$
 Then $(C,d_C)$ is a geodesic $s$-$t$ graph, with $s(C)=x_0$ and $t(C)=x_m$.
 
  We call  in that case $(C, d_C)$ a \textit{geodesic $s$-$t$ cycle}.

 \item[c)] {\em Generalized Laakso graphs}. For $k,l_1,l_2,m\in \N_0$ such that $l_1\geq l_2\geq 1$, and $l_1+l_2\ge 3$. Let $L=(V(L), E(L))$  be  defined by
 $$V(L)=\{x_i\}_{i=0}^k\cup\{ y^{(1)}_i\}_{i=0}^{l_1}\cup \{ y^{(2)}_i\}_{i=0}^{l_2}\cup \{z_i\}_{i=0}^m,$$
 where we make the following identifications:
 $$x_k\equiv y^{(1)}_0\equiv y^{(2)}_0,\text{ and }  y^{(1)}_{l_1}\equiv y^{(2)}_{l_2}\equiv z_0,$$
 \noindent and 
 \begin{align*} E(L)&=\big\{\{x_{i-1}, x_i\}: i=1,2,\ldots,k\big\} \cup \big\{\{z_{i-1}, z_i\}: i=1,2,\ldots,m\big\}\\
                               &\qquad   \cup  \big\{\{y^{(1)}_{i-1}, y^{(1)}_i\}: i=1,2,\ldots,l_1\big\}    \cup  \big\{\{y^{(2)}_{i-1}, y^{(2)}_i\}: i=1,2,\ldots,l_2\big\}.
                                  \end{align*}

We call $L$ a $(k,l_1,l_2,m)$-\textit{Laakso graph}. If $L$ is a $(k,l_1,l_2,m)$-Laakso graph for some $k,l_1,l_2,m$, then we say that $L$ is \textit{a generalized Laakso graph}. Note that in particular, a $(1,2,2,1)$-Laakso graph is the base graph for the family of standard Laakso graphs, while a $(0,2,2,0)$-Laakso graph is the base graph for the family of standard diamond graphs.  We denote $s(L)=x_0$ and $t(L)=z_m$ and note that the orientation 
\begin{align*} E_d(L)&=\big\{(x_{i-1}, x_i): i=1,2,\ldots,k\big\} \cup \big\{(z_{i-1}, z_i): i=1,2,\ldots,m\big\}\\
                               &\qquad   \cup  \big\{(y^{(1)}_{i-1}, y^{(1)}_i): i=1,2,\ldots,l_1\big\}    \cup  \big\{(y^{(2)}_{i-1}, y^{(2)}_i): i=1,2,\ldots,l_2\big\}.
                                  \end{align*}

implies that these generalized Laakso graphs are $s$-$t$ graphs.

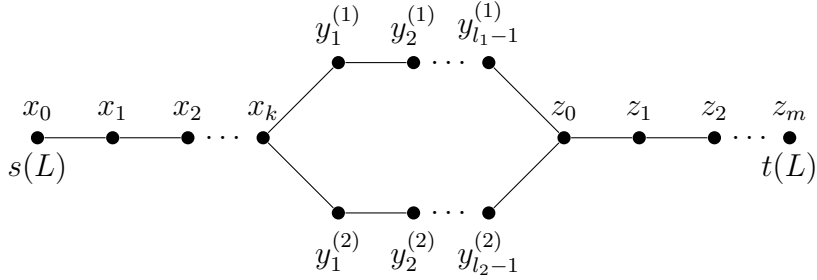
\begin{figure}[h]
\centering
\begin{tikzpicture}[roundnode/.style={circle, fill=black, inner sep=0pt, minimum size=1.75mm}]
  \node[roundnode, label=90:{$x_0$}] (x0) at (0,0) {};
        \node[roundnode, label=90:{$x_1$}] (x1) at (1,0) {};
        \node[roundnode, label=90:{$x_2$}] (x2) at (2,0) {};
        \node (p1) at (2.5,0) {\ldots};
        \node[roundnode, label=90:{$x_k$}] (xk) at (3,0) {};

        \node[roundnode, label=90:{$y_1^{(1)}$}] (y11) at (4,1) {};
        \node[roundnode, label=90:{$y_2^{(1)}$}] (y21) at (5,1) {};
        \node (p2) at (5.5,1) {\ldots};
        \node[roundnode, label=90:{$y_{l_1-1}^{(1)}$}] (yl'1) at (6,1) {};

        \node[roundnode, label=270:{$y_1^{(2)}$}] (y12) at (4,-1) {};
         \node[roundnode, label=270:{$y_2^{(2)}$}] (y22) at (5,-1) {};
        \node (p3) at (5.5,-1) {\ldots};
        \node[roundnode, label=270:{$y_{l_2-1}^{(2)}$}] (yl'2) at (6,-1) {};

        \node[roundnode, label=90:{$z_0$}] (z0) at (7,0) {};
        \node[roundnode, label=90:{$z_1$}] (z1) at (8,0) {};
        \node[roundnode, label=90:{$z_2$}] (z2) at (9,0) {};
        \node (p4) at (9.5,0) {\ldots};
        \node[roundnode, label=90:{$z_m$}] (zm) at (10,0) {};

        \node (s) at (0,-.4) {$s(L)$};
        \node (t) at (10,-.4) {$t(L)$};

        \draw[-] (x0) -- (x1);
        \draw[-] (x1) -- (x2);
        \draw[-] (xk) -- (y11);
        \draw[-] (xk) -- (y12);
        \draw[-] (y11) -- (y21);
        \draw[-] (y12) -- (y22);
        \draw[-] (yl'1) -- (z0);
        \draw[-] (yl'2) -- (z0);
        \draw[-] (z0) -- (z1);
        \draw[-] (z1) -- (z2);

        \end{tikzpicture}
        \caption{A $(k,l_1,l_2,m)$-Laakso Graph}
        \end{figure}

 Let $w:E(L)\rightarrow (0,1]$ be such that,
      \begin{align*}&\sum_{i=1}^k w(x_{i-1},x_i)+  \sum_{i=1}^{l_1} w(y_{i-1}^{(1)},y_i^{(1)})+\sum_{i=1}^m w(z_{i-1},z_i)   \\
                   &\qquad=\sum_{i=1}^k w(x_{i-1},x_i)+  \sum_{i=1}^{l_2} w(y_{i-1}^{(2)},y_i^{(1)})+\sum_{i=1}^m w(z_{i-1},z_i)\end{align*}
         then $w$ generates a geodesic metric $d_L$ on $L$ which turns  $(L,d_L)$ into  a  geodesic $s$-$t$ graph.
\end{enumerate}

We call a $(k,l_1,l_2,m)$-Laakso graph $L=(V(L),E(L)) $ {\em balanced } if $l_1=l_2$.
\end{exs}

An important attribute of geodesic $s$-$t$ graphs is that all cycles and geodesic $s$-$t$ subgraphs over the same two distinguished points $s,t$ are isometric subgraphs. Such a property will allow us to restrict our attention to \textit{generalized Laakso graphs} as introduced in Example \ref{Ex:2.3.1} (c).

For the next result, we introduce the following notation: If $G$ is an $s$-$t$ graph,  we call a subgraph $H$ of $G$ which contains 
$s(G)$ and $t(G)$ and which has the property that 
it is an $s$-$t$ graph with the same distinguished points $s(G)$ and $t(G)$,
an {\em  $s$-$t$ subgraph of $G$}.

\begin{lem}\label{L:2.3.2} Suppose $(G,d_G)$ is a geodesic $s$-$t$ graph with distinguished points $s(G)$ and $t(G)$.

\begin{enumerate}
  \item  Every cycle $C$ in $G$ is a subgraph of a generalized Laakso graph $L$  in $G$,
  for which $s(L)=s(G)$ and $t(L)=t(G)$ and both, $(C, d_C)$ as well as  $(L,d_L)$ are isometric subgraphs of $G$, where
  $d_C$ and $d_L$ are the geodesic metrics on $C$ and $L$, respectively, induced by $d_G$.

    \item If $H$ is  an $s$-$t$ subgraph of $G$   then  $(H,d_H)$ is an isometric subgraph of $(G,d_G)$, where $d_H$ is the geodesic metric on $H$ induced 
    by $d_G$.
 \end{enumerate}

Moreover, if $G$ is not a path, it contains cycles and, thus, also a generalized Laakso graph.
\end{lem}
\begin{proof}
Note that as every vertex of $G$ is  on an $s$-$t$ path, if $G$ is a tree, then it must simply be an $s$-$t$ path. 
This proves the ``moreover'' part of our claim. 

\noindent Proof of (1). Let $C\subset G$ be an arbitrary cycle. Label the vertices of $C$ by $C=(x_0,x_1,\ldots, x_n=x_0)$ and suppose that $d_C$ is the induced geodesic metric of $d_G$ on $C$.
We choose $y$ and $z$ in $ V(C)$ such that 
$$d_G(y,s(G))=\min_{x\in V(C)} d_G(x,s(G)), \text{ and } d_G(z,t(G))=\min_{x\in V(C)} d_G(x,t(G)),$$
and observe that $y\not=z$. Indeed, assume that $y=z$. Let $y'\in V(C)$ so that $\{y,y'\}\in E(C)\subset E(G)$, and let $P=(x_j)_{j=0}^n$ be a path between $s(G)$ and $t(G)$ containing $\{y,y'\}$.
It follows that $\length_{d_G}(P)=d_G(s(G), t(G))$. We assume without loss of generality  that  $x_{i_1}=y$ and $x_{i_2}=y'$, with $0<i_1<i_2<n$ 
(if $i_2<i_1$ we swap $s(G)$ with $t(G)$). Then
\begin{align*}
d_G(s(G),y)+d_G(y, t(G))&\ge d_G(s(G), t(G))
=\length_{d_G}(P)\\
&=d_G(s(G),y)+d_G(y,y')+d_G(y', t(G))\\
&> d_G(s(G),y)+d_G(y', t(G))\ge d_G(s(G),y)+d_G(y, t(G)),
\end{align*}
where the last inequality follows from the minimality conditions on $y=z$.
This is a contradiction; thus, we conclude that $y\not=z$.

After relabeling the vertices of $C$ we can assume that $y=x_0$ and $z=x_{l_1}$, for some $0<l_1<n$.
We let $A=(a_j)_{j=0}^k$ be a path between $a_0=s(G)$ and $a_k=x_0=y$, of shortest metric length, 
$B=(b_j)_{j=0}^m$ be a path between $b_0=x_{l_1}=z$ and $b_m=t(G)$, of shortest metric length,
and let $Q^{(1)}=(x_j)_{j=0}^{l_1}$ and $Q^{(2)}=(x_{n-j})_{j=0}^{n-l_1}$ (both being paths from $y$ to $z$, and together forming the cycle $C$). Since $V(A)\cap V(Q^{(1)})=V(A)\cap V(Q^{(2)})=\{x_0\}$ and $V(B)\cap V(Q^{(1)})=V(B)\cap V(Q^{(2)})=\{x_{l_1}\}$
it follows  that the graph  $L$
with $V(L)=\{a_j\}_{j=0}^k\cup\{ x_j\}_{j=0}^{l_1} \cup \{x_{n-j}\}_{j=0}^{n-l_1}\cup \{b_j\}_{j=0}^m$ and $E(L)=E(A)\cup E(Q^{(1)})\cup E(Q^{(2)})\cup E(B)$
is a $(k,l_1,n-l_1,m)$- Laakso graph. Let $d_L$ and $d_C$ be the induced geodesic metrics on $V(L)$ and $V(C)$, respectively.
We need to show that they coincide with the restriction of $d_G$ to $V(L)$ and $V(C)$, respectively.

 We define $P^{(1)}= A\smile Q^{(1)} \smile B$ and  $P^{(2)}= A\smile Q^{(2)} \smile B$, and recall that by Proposition \ref{P:2.3.0} they are both isometric subpaths of $(G, d_G)$ as well  as  $(L,d_L)$. Since they have the same length it follows that $\length_{d_G}(Q^{(1)})=\length_{d_G}(Q^{(2)})$.
 
  Let $u,v\in V(L)$. We need to show that $d_L(u,v)= d_G(u,v)$, and in case that $u,v\in V(C)$, also  that $d_C(u,v)= d_G(u,v)$.
We first consider the case that  $u,v$ are both in $P^{(i)}$  for some $i=1,2$.
Then it follows that $d_G(u,v)$ and $d_L(u,v)$ are both equal to the induced metric on $P^{(i)}$, and, thus, equal to each other. If moreover  $u,v\in V(C)$, and  thus  $u,v\in V(Q^{(i)})$,
the path of shortest length in $C$ between $x$ and $y$ (with respect to $d_C$!) will be inside $Q^{(i)}$. Thus, again by  Proposition \ref{P:2.3.0}, it follows that $d_C(u,v)= d_L(u,v)=d_G(u,v)$.

If  $u,v$ are not  both in $P^{(i)}$, we can assume that $u=x_{i_1}$ for some $1<i_1<l_1$ and $v=x_{i_2}$ for some $l_1<i_2<n$.
Let us assume that $d_C(u,v)\not=d_G(u,v)$, 
and thus since always  $d_C(u,v)\ge d_G(u,v)$, that  $d_C(u,v)> d_G(u,v)$. This means that there is a path $P=(w_j)_{j=0}^n$
in $G$ from $w_0=x_{i_1}$ to $w_n=x_{i_2}$ in $G$ whose metric length is smaller than $d_C(u,v)$.
We can also assume that $u\in V(Q^{(1)})$ and $v\in V(Q^{(2)})$ are chosen, so that $n$ is minimal,
which implies that $w_j\in V(G)\setminus V(C)$ for $j=1,2,\ldots, n-1$. Let $P'=(w_{n-j})_{j=0}^n$, 
the path from $x_{i_2}$ to $x_{i_1}$, which reverses $P$. 

It follows that 
\begin{align*} &R^{(1)}= A\smile (x_j)_{j=0}^{i_1}\smile P \smile (x_{i_2-j})_{j=0}^{i_2-l_1} \smile B
\text{ and }\\ &R^{(2)}= A\smile (x_{n-j})_{j=0}^{n-i_2}\smile P' \smile (x_{j})_{j=i_1}^{l_1} \smile B\end{align*}
are both paths from $s(G)$ to $t(G)$ and therefore of  metric length $d_G(s(G),t(G))$.
But on the other hand we have, by Proposition \ref{P:2.3.0}
\begin{align*}
\length_{d_G}(R^{(1)})+\length_{d_G}(R^{(2)})=& 2\length_{d_G}(A) + 2\length_{d_G}(B) + 2\length_{d_G}(P)\\
 &+\length_{d_G}\big((x_j)_{j=0}^{i_1}\big)+\length_{d_G}\big((x_j)_{j=i_1}^{l_1}\big)\\
 &+\length_{d_G}\big((x_{n-j})_{j=0}^{n-i_2}\big)+\length_{d_G}\big((x_{i_2-j})_{j=0}^{i_2-l_1}\big)\\
 &=2 d_G(s(G),t(G))+  2\length_{d_G}(P),
\end{align*}
which would mean that $\length_{d_G}(P)=0$, and is thus a contradiction. We deduce that $d_C(u,v)=d_G(u,v)$
and, thus, since $d_G(u,v)\le d_L(u,v)\le d_C(u,v)$, also $d_L(u,v)=d_C(u,v)$.

\noindent Proof of (2). Assume that $H=(V(H),E(H))$ is a subgraph of $G$, which is an $s$-$t$ graph with $s(H)=s(G)$.
 Let $d_H$, be  the  geodesic metric on $V(H)$, induced by $d_G$.
 Since every path in $H$ is a path in $G$, it follows that $H$ together with $d_H$  is a geodesic $s$-$t$ graph. 
 Thus, if   $u,v\in V(H)$ lie on the same path  $Q$ in $H$ from $s(G)$ to $t(G)$ that  is also a path in $G$,
  then $d_G(u,v)=d_Q(u,v) =d_H(u,v)$.
  If  for $u,v\in V(H)$ there is no path in $H$ which contain $u$ and $v$, we can find two paths
  from $s(G)$ to $t(G)$ in $H$,  $Q^{(1)}=(x_i)_{i=0}^{l_1}$ and $Q^{(2)}=(y_j)_{j=0}^{l_2}$, so that  $Q^{(1)}$ contains $u$ but not  $v$ and
  $Q^{(2)}$ contains $v$ but not $u$. Let  $i_1\in\{1,2,\ldots,l_1-1\}$ and $j_1=\{1,2, \ldots,l_2-1\}$ such that $x_{i_1} =u$ and $y_{j_1}=v$.
  We define
  \begin{align*}
  i_0&=\max \big\{0\le i<i_1:x_i\in \{y_j\}_{j=0}^{j_1-1}\big\} \text{ and } j_0=\max \big\{0\le j<j_1:y_j\in \{x_i\}_{i=0}^{i_1-1}\big\}
  \intertext{and}
  i_2&=\min\big\{ i_1<i\le l_1: x_i\in \{y_j\}_{j=j_1+1}^l\big\} \text{ and } j_2=\min\big\{j_1<j\le l_2: y_j\in \{x_i\}_{i=i_1+1}^k\big\}.
  \end{align*}
Then it follows that $x_{i_0}=y_{j_0}$ and $x_{i_2}=y_{j_2}$, and thus
$$C=(x_{i_0}, x_{i_0+1},\ldots, x_{i_2}=y_{j_2}, y_{j_2-1}, y_{j_2-2}, \ldots y_{j_0}=x_{i_0} )$$
is a cycle in $H$.  Denoting by $d_C$ the  geodesic metric on $C$ induced by $G$,
which is the same metric induced by $d_H$ (because they are both generated by the same weights), we deduce that from our previous work that 
$d_H(u,v)= d_C(u,v)=d_G(u,v)$.
 \end{proof}

 \medskip\noindent 2.4. {\bf Measured geodesic $s$-$t$ graphs.}\label{SS:2.4}
A triple  $(G,d_G,\nu_G)$, with $(G,d_G)$ being a  geodesic $s$-$t$ graph, and $\nu_G$
being a probability measure on $E(G)$,  is called a {\em measured  geodesic $s$-$t$ graph}.

\begin{exs}\label{Ex:2.4.1} For the three elementary geodesic  $s$-$t$  graphs, introduced in Examples \ref{Ex:2.3.1},
we choose the following probabilities on their edges.
\begin{enumerate}
\item[a)] Let $P=(x_j)_{j=0}^k$ be a path with a geodesic  metric $d_P$ generated by a weight function $w:E(P)\to \R^+$.
Then put $$\nu_P(e)=\frac{w(e)}{\sum_{j=1}^k w(\{x_{j-1},x_j\})}=\frac{d_P(e)}{d_P(s(P),t(P))}. $$
\item[b)]  Let $C=(x_0,x_1,\ldots, x_n=x_0)$  together with a geodesic metric  $d_C$ a geodesic $s$-$t$ cycle, and let $1\le m\le n-1$
be such that 
\begin{equation*}d_C(x_0,x_m)=\sum_{j=1}^m d_C(x_{j-1},x_j)=\sum_{j=m+1}^n d_C(x_{j-1},x_j)=\frac12 \sum_{j=1}^n d_C(x_{j-1},x_j).\end{equation*}
and let $s(C)=x_0$, $t(C)=x_m$. Then we put 
$$\nu_C(e)=\frac12 \frac{ d_C(e)}{d_C(x_0,x_m)}.$$
\item[c)] Let $L=(V(L), E(L))$ be a $(k,l_1,l_2,m)$-Laakso graph,
with $$V(L)=\{x_i\}_{i=0}^k\cup\{ y^{(1)}_i\}_{i=0}^{l_1}\cup \{ y^{(2)}_i\}_{i=0}^{l_2}\cup \{z_i\}_{i=0}^m,$$
where
 $$x_k\equiv y^{(1)}_0\equiv y^{(2)}_0,\text{ and }  y^{(1)}_{l_1}\equiv y^{(2)}_{l_2}\equiv z_0.$$
 \begin{align*} E(L)&=\big\{\{x_{i-1}, x_i\}: i=1,2,\ldots,k\big\} \cup \big\{\{z_{i-1}, z_i\}: i=1,2,\ldots,m\big\}\\
                               &\qquad   \cup  \big\{\{y^{(1)}_{i-1}, y^{(1)}_i\}: i=1,2,\ldots,l_1\big\}    \cup  \big\{\{y^{(2)}_{i-1}, y^{(2)}_i\}: i=1,2,\ldots,l_2\big\}.
                                  \end{align*}
                               For a weight function $w:E(L)\rightarrow\R^+$ satisfying 
      \begin{align*}&\sum_{i=1}^k w(x_{i-1},x_i)+  \sum_{i=1}^{l_1} w(y_{i-1}^{(1)},y_i^{(1)})+\sum_{i=1}^m w(z_{i-1},z_i)   \\
                   &\qquad=\sum_{i=1}^k w(x_{i-1},x_i)+  \sum_{i=1}^{l_2} w(y_{i-1}^{(2)},y_i^{(1)})+\sum_{i=1}^m w(z_{i-1},z_i)\end{align*}
                   which generates a geodesic metric $d_L$
turning $L$ into a geodesic $s$-$t$ graph, with $s(L)=x_0$ and $t(L)=z_m$  we define $\nu_L$ by
 \begin{align*}
 \nu_L(\{x_{i-1}, x_i\}) &= \frac{d_L(x_{i-1},x_i)}{d_L(s(L),t(L))} \text{ for $i=1,2\ldots k$,}\\
 \nu_L(\{ y^{(1)}_{i-1}, y^{(1)}_i\})&=\frac12\frac{d_L(y^{(1)}_{i-1},d_L(y^{(1)}_i))}{d_L(s(L),t(L))} \text{ for $i=1,2\ldots l_1$,}\\
  \nu_L(\{ y^{(2)}_{i-1}, y^{(2)}_i\})&=\frac12\frac{d_L(y^{(2)}_{i-1},d_L(y^{(2)}_i))}{d_L(s(L),t(L))} \text{ for $i=1,2\ldots l_2$,}\\
\nu_L(\{z_{i-1}, z_i\}) &= \frac{d_L(z_{i-1},z_i)}{d_L(s(L),t(L))} \text{ for $i=1,2,\ldots,m$}.
\end{align*}
\end{enumerate}

\end{exs}

\noindent 2.5. \textbf{Expected Distortion} Assume that $G$ and $H$ are simple, connected graphs and $d_G$,$d_H$ are geodesic metrics on $G$ and $H$, respectively. If $\nu_G$ is a probability on $E(G)$ and $f: V(G)\to V(H)$ then  the {\em expected distortion of   $f$ with respect to $(G,d_G,\nu_G)$} is defined by
$$\D_{\nu} (f)=\E_\nu\Big(\frac{d_H\circ f}{d_G}\Big)=\sum_{e\in E(G)} \frac{d_H(f(e))}{d_G(e)} \nu(e).$$

\noindent 2.6. \textbf{Stochastic Embeddings} Let $\cM$ be a class of metric spaces, and let  $(X,d_X)$  be a  metric space.
  A family $(f_i)_{i=1}^n$ of maps $f_i:X\to M_i$, with $(M_i,d_i)\in\cM$, together with numbers $\P=(p_i)_{i=1}^n\subset[0,1]$, such that $\sum_{i=}^np_i=1$,    is called a 
  {\em $D$-stochastic embedding of $X$ into elements of the class  $\cM$} if  for all $x,y\in X$ and $i=1,\ldots,n$
  \begin{align}
   &d_X(x,y)\le  d_i(f_i(x),f_i(y)), \notag\\
   &\E_\P\big(d_i(f_i(x),f_i(y)\big)=\sum_{i=1}^n p_i d_i\big(f_i(x),f_i(y) \big)\le D d_X(x,y). \notag
  \end{align} 
  In that case we say that $(X,d_X)$  is {\em $D$-stochastically  embeddable  into   $\cM$}. We denote
  $$\cS_{\cM}(X,d_X):=\inf\{D\geq 1:\ (X,d_X)\ \text{is}\ D-\text{stochastically embeddable into}\ \cM\}$$
  \noindent and call the value $\cS_{\cM}(X,d_X)$ \textit{the stochastic distortion of} $(X,d_X)$ into $\cM$.

By following the same ideas as in \cite{AlonKarpPelegWest1995} one could deduce from the Min-Max Theorem 
that 
$$\cS_\cM(X,d_X)= \sup_\nu \inf_{(f,M)} \sum_{x,y\in X, x\not=y} \nu(\{x,y\}) \frac{d_M(f(x),f(y))} {d_X(x,y)}, $$
where the $\sup$ is taken over all probabilities on the doubletons of $X$ with finite support,
and the $\inf$ is take over all pairs $(f,M)$ with $(M,d_M)$ being in the class $\cM$ and $f: (X,d_X)\to (M, d_M)$ being expansive.
We will only need that the left side is not smaller than the right side, and this only for geodesic graphs, a result 
which is much more elementary and will be observed in the next Proposition.

\begin{prop}\label{P:2.6.1}  Let $(G,d_G)$,  be a geodesic  graph, and let $\nu_G$ be a probability on $E(G)$. Suppose that $\cM$ is a family of geodesic graphs.

If $(H_i, d_{H_i})\in \cM$, for $i=1,2,\ldots,n$,
and
if $(f_i)_{i=1}^n$, $f_i:V(G)\to V(H_i)$, together with $\P=(p_i)_{i=1}^n\subset (0,1]$ forms a $D$-stochastic embedding 
of $G$ into $(H_i)_{i=1}^n$,
then $D\ge c_\nu(G,d_G)$, where 
$$c_\nu(G,d_G):=
\inf \left\{ \D_{\nu} (f)\Big\vert  (H,d_H)\in \cM,
                                        f:(V(G),d_G)\to (V(H), d_H)\text{ expansive} \right\},$$
and, thus, 
$$D\ge  c(G,d_G):=\max\big\{ c_\nu(G,d_G): \nu\text{ is a probability on $E(G)$}\big\}.$$
\end{prop}
\begin{proof} We observe for a probability $\nu$ on $E(G)$ that 
\begin{align}
D&\ge \max_{x,y\in V(G), x\not= y}\frac{\sum_{i=1}^n p_i d_{H_i}(f_i(x),f_i(y))}{d_G(x,y)}\notag\\
&\ge \max_{e=\{x,y\}\in E(G)} \frac{\sum_{i=1}^n p_i d_{H_i}(f_i(x),f_i(y))}{d_G(x,y)}\notag\\
&\ge\sum_{e=\{x,y\}\in E(G)} \frac{\nu(e)}{d_G(x,y)}\sum_{i=1}^n p_i d_{H_i}(f_i(x),f_i(y))\notag\\
&= \sum_{i=1}^n p_i \sum_{e=\{x,y\}\in E(G)} \frac{\nu(e)}{d_G(x,y)}d_{H_i}(f_i(x),f_i(y))\notag \\
&=\sum_{i=1}^n p_i\D_{\nu} (f_i) \ge c_\nu(G,d_G).\notag 
\end{align}
\end{proof}

If $(X,d_X)$ is a finite metric space with $|X|=n$, then by a result of Fakcharoenphol, Rao, and Talwar a stochastic embeddings with $O(\log(n))$ stochastic distortion into the family of geodesic trees always exists.

\begin{thm}\label{T:2.6.1}\cite{FakcharoenpholRaoTalwar2004}
    If $(X,d_X)$ is a finite metric space with $|X|=n$ and $\cT$ denotes the family of weighted trees with the corresponding induced geodesic metrics, then
    $$S_{\cT}(X,d_X)=O(\log(n))$$
\end{thm}

 Thus, in light of Theorem \ref{T:2.6.1}, in proving the Main Theorem it suffices to give a logarithmic lower bound for slash powers of geodesic $s$-$t$ graphs that contain a cycle. 
 We accomplish this by calculating the expected distortion with respect to a probability defined on the edges.

\noindent We shall utilize Proposition \ref{P:2.6.1} in the case that $\cM$ is the family of geodesic trees. \\

\medskip\noindent 2.7. {\bf Slash Products  of $s$-$t$ Graphs}\label{SS:2.7}
(cf.  \cite{LeeRaghavendra2010}).
Assume $G$ and $H$ are directed $s$-$t$ graphs (\ie\  directed graphs whose edges are oriented in a way so that every edge is on a directed path from $s(G)$ to $t(G)$ and $s(H)$ to $t(H)$ respectively).

The vertices of $H\oslash G$ are defined by

$$V(H\oslash G)=\{ (e,v): e\in E_d(H), v\in V(G)\},$$ where we identify
  the elements $(e,t(G))$ and $(\tilde e,s(G))$, for $e,\tilde e\in E_d(H)$, for which $e^+=\tilde e^-$.
  We also consider $V(H)$ to be a subset of $V(H\oslash G)$, by making 
  for $u\in V(H)$ the following identification:
  \begin{align*}
   &u \equiv  (e,s(G)),  \text{ whenever $\tilde e \in E_d(H)$, for which $e^-=u$, and }\\
  &u\equiv (\tilde e,t(G)), \text{ whenever $\tilde e \in E_d(H)$, for which $\tilde e^+=u$. }
  \end{align*}
  In particular
  \begin{align*}
  &s(H)\equiv (e,s(G)),  \text{ whenever $e\in E_d(H)$, with $e^-=s(H)$, and }\\
  &t(H)\equiv (\tilde e,t(G)),  \text{ whenever $\tilde e\in E_d(H)$, with $\tilde e^+=t(H)$.} \end{align*}

\noindent  The directed edges of $H\oslash G$ are defined by:
 \begin{align*}E_d(H\oslash G)&=\big\{ \big( (e,f^-),(e,f^+)\big): e\in E_d(H), f\in E_d(G)\big\}\end{align*}
 the undirected edges 
 by 
  \begin{align*}E(H\oslash G)&=\big\{ \big\{ (e,f^-),(e,f^+)\big\}:e\in E(H), f\in E_d(G)\big\}\\
  &=\big\{ \big\{(e,u),(e,v)\big\}:e\in E(H), \{u,v\}\in E(G)\big\}.\end{align*}
 We abbreviate for $e\in E_d(H)$, $f\in E_d(G)$, the edge $\big( (e,f^-),(e,f^+)\big)$ (of $H\oslash G$), by
$e\oslash f$, and for $v\in V(G)$ the vertex $(e,v)$ by $e\oslash v$.

\begin{rem} Note that for $e\in E(H)$, 
$$e\oslash G= ( e\oslash V(G), e\oslash E(G))=\big(\{(e\oslash v): v\in V(G)\}, \{(e\oslash f): f\in E(G)\}\big)$$
is a subgraph of $H\oslash G$, which is graph isomorphic to $G$, via the embedding
$\psi_e:V(G)\to V(H\oslash G),\quad v\mapsto e\oslash v$.
\end{rem}

Therefore, on may think of $H\oslash G$ to be obtained by replacing each edge of $H$ by a copy of $G$. More formally,
if $H$ is a directed  graph $e\in E_d(H)$ and if $G$ is an $s$-$t$ graph, then we mean by the graph
 {\em obtained by replacing the edge $e$  of $H$ by a copy of $G$} the graph
 $H{^e\hskip-7pt-}G$, defined by
 \begin{align*}
 V(H{^e\hskip-7pt-}G)&=V(H)\cup V(G),\\ &\quad \text{ where we identify $e^-\equiv s(G)$ and $e^+\equiv t(G)$,}\\
 E(H{^e\hskip-7pt-}G)&= \big(E(H)\setminus\{e\}\big)\cup  E(G).\end{align*}

 Using the identification of vertices    of  a graph $H$ with elements of a slash product $H\oslash G$ we observe the following proposition.
   \begin{prop}\label{P: 2.7.1} If $H$ and $G$ are $s$-$t$ graphs then 
 $H\oslash G$ is also  an $s$-$t$ graph, with $s(H\oslash G)= s(H)$ and  
  $t(H\oslash G)= t(H)$. 
  \end{prop}
  \begin{proof}
 Note that every  path from   $s(H)$ to  $t(H)$ in $H\oslash G$ is a  path which is obtained
  from a path $(y_i)_{i=0}^l$ in $H$ from   $s(H)$ to  $t(H)$, by replacing each edge $e_i=\{y_{i-1}, y_i\}$  
  by a path in $H\oslash G$ of the form $\big(e_i,x^{(i)}_j \big)_{j=0}^{l_i}$, where $(x^{(i)}_j)_{j=0}^{l_i}$ is path from $s(G)$ to $t(G)$ in $G$,  for $i=1,2\ldots, l$. 
\end{proof}
Assume  that $(H, d_H)$ and $(G,d_G)$ are normalized geodesic $s$-$t$ graphs. On $V(H\oslash G)$ we let $d_{H\oslash G}$ be the geodesic metric on $V(H\oslash G)$ generated by the weight function defined on all $e\oslash f\in E(H\oslash G)$ by
$$w_{H\oslash G}(e\oslash f)= d_H(e)\cdot d_G(f).$$
It follows then that $(H\oslash G, d_{H\oslash G})$ is a normalized geodesic $s$-$t$ graph.
Assume, moreover that  $\nu_G$ and $\nu_H$ are probabilities on $E(G)$ and $E(H)$, respectively.
We consider the {\em product probability $\nu_{H\oslash G}$ on $E_d(H\oslash G)$}, which is given by 
$$\nu_{H\oslash G} (e\oslash f)= \nu_H(e)\cdot\nu_G(f),\text{ for $e\in E(H)$ and $f\in E(G)$.}$$
It follows then that $(H\oslash G, d_{H\oslash G}, \nu_{H\oslash G})$ is a measured normalized geodesic $s$-$t$ graph.

We state some easy to prove and well known facts about slash products:
\begin{prop}\label{P:2.7.1}{\rm  \cite{LeeRaghavendra2010}*{Lemma 2.3}}
Taking slash products is an associative operation, i.e. if
$G$, $H$ and $K$ are $s$-$t$ graphs then 
Then 
\begin{align*}
\Psi: &V\big(K\oslash (H\oslash G)\big)\to V\big((K\oslash H)\oslash G\big),\\
&\quad e\oslash(f\oslash v)\mapsto (e\oslash f)\oslash v, \text{ for $e\in E(K)$, $f\in E(H)$, and $v\in V(G)$.}
 \end{align*}
 is a graph isomorphism. Identifying graphs  $K\oslash (H\oslash G)$ and  $(K\oslash H)\oslash G$ we denote them simply by $K\oslash H\oslash G$
and $e\oslash f\oslash g$ instead of $(e\oslash f)\oslash g$ and $e\oslash( f\oslash g)$.

 Moreover, if $(G,d_G)$, $(H,d_H)$ and $(K, d_K)$ are geodesic $s$-$t$ graphs then 
 $$\Psi \big(V\big(K\oslash (H\oslash G)\big), d_{K\oslash (H\oslash G)}\big)\to \big(V\big((K\oslash H)\oslash G\big), d_{(K\oslash H)\oslash G}\big)$$ 
 is also a  metric isometry, and 
 \begin{align*}
 \nu_{K\oslash (H\oslash G)}(e\oslash(f\oslash g))&= \nu_{(K\oslash H)\oslash G}((e\oslash f)\oslash g)= \nu_K(e)\nu_H(f)\nu_G(g), \\
 &\text{ for $e\in E(K)$, $f\in E(H)$, and $g\in E(G)$.}
 \end{align*}
\end{prop}

For an $s$-$t$ graph $G$ we define the {\em $n$-th  slash power of $G$}, by induction: $G^{\oslash 1}=G$, and assuming
$G^{ \oslash n}$ is defined we put $G^{\oslash(n+1)}=G\oslash G^{\oslash n}$. It follows from the above  remark  that $G^{\oslash n}$ is also an $s$-$t$ graph
with $s(G^{\oslash n})=s(G)$ and  $t(G^{\oslash n})=t(G)$ (using the identification of $V(G)$ with a subset of $V(G^{\oslash n})$ introduced in the previous subsection).
Using the graph-isomorphisms defined in Proposition  \ref{P:2.7.1},
we can write the elements  $e\in E(G^{\oslash n})$ as 
$e=e_n\oslash e'$,  as $e=\tilde e\oslash e_1$, or as $e=e_1\oslash e_2\oslash \ldots e_n$, with $e_j\in E(G)$, for $j=1,2\ldots, n$, and $e',\tilde e\in E(G^{\oslash(n-1)})$.
We can write the elements $v\in V(G^{\oslash n})$ as
$v=e_n\oslash w$, or as $e'\oslash z$, with $e_n\in E(G)$, $w\in V(G^{\oslash(n-1)})$, $e'\in E(G^{\oslash (n-1)})$ and $z\in V(G)$.
If $(G,d_G,\nu_G)$ is a measured normalized metric $s$-$t$ graph, we define 
metric $d_{G^{\oslash n}}$ on   $V(G^{\oslash n})$ and the probability $\nu_{G^{\oslash n}}$ also inductively 
by $d_{G^{\oslash 1}}=d_G$, $\nu_{G^{\oslash 1}}=\nu_G$, 
 and $ d_{G^{\oslash n}}(e)=d_G(e_n) d_{G^{\oslash (n-1)}}(e')$ and 
 $ \nu_{G^{\oslash n}}(e)=\nu_G(e_n) \nu_{G^{\oslash (n-1)}}(e')$, for  $e=e_n\oslash e'$, $e'\in E(G^{\oslash (n-1)})$, and $e_n\in E(G)$.\\

\begin{prop}\label{P:2.7.3} Let $(G,d_G)$ and $(H,d_H)$ be normalized geodesic $s$-$t$ graphs, and 
let $G'$ and $H'$ be $s$-$t$ subgraphs of $ H\oslash G$, respectively.
Then $H'\oslash G'$ is an  $s$-$t$ subgraph of $ H\oslash G$, and thus, by Lemma \ref{L:2.3.2}, it is an isometric subgraph of 
$H\oslash G$.

In particular, for $n\in\N$,  $G'^{\oslash n}$ is an $s$-$t$, and thus isometric, subgraph of $G^{\oslash n}$.
\end{prop}
\begin{proof} The claim follows from the fact that every $s$-$t$ path in  $H'\oslash G'$, is a
concatination of paths  $P_i= \big((x_{i-1},x_i)\oslash y^{(i)}_j\big)_{j=0}^{l_i}$, $i=1,2,\ldots l$,
where $(y_j^{(i)})_{j=0}^{l_i}$ is a path in $G'$ from $y^{(i)}_0=s(G)$ to $y_{l_i}^{(i)}=t(G)$ in $G'$, for 
every $i=1,2\ldots, l$, and $(x_i)_{i=0}^l$  is a path in $H'$ for $s(H)$ to $t(H)$.

Since the metric on $H'\oslash G'$ induced by $d_H\oslash d_G$ is the geodesic  metric induced by the  weight function 
$w': E(H'\oslash G')\to \R^+$, $e\oslash f\mapsto d_H(e)\cdot d_G(f)$, Lemma \ref{L:2.3.2} implies that 
$H'\oslash G'$ is an isometric subgraph of $H\oslash G$.
\end{proof} 

The following shows that a high enough slash power of a generalized Laakso graph will eventually contain a balanced Laakso subgraph. This will allow us to only consider the balanced case in our analysis in Section 4. 

\begin{lem}\label{L:2.7.4} Let $L=(V(L),E(L))$ be  a  $(k,l_1,l_2,m)$-Laakso graph, with $l_1< l_2$.
Put
$$n_0=2+\Bigg\lfloor\frac {\ln (l_2) -\ln (l_1)}{\ln(k+m+l_2)-\ln(k+m+l_1)}\Bigg\rfloor.$$
Then $(L^{\oslash n_0},d_{L^{\oslash n_0}} )$  contains a subgraph which is a balanced generalized Laakso graph $L'$  with $s(L')=s(L^{\oslash n_0})$
and  $t(L')=t(L^{\oslash n_0})$. Moreover, if $(L,d_L)$ is also a normalized geodesic $s$-$t$ graph,  $c_0$ being  the  metric length of its cycle,
then  $L'$ with the induced geodesic metric $d_{L'}$ on edges of $L'$ will be an isometric geodesic subgraph of
 $(L^{\oslash n_0},d_{L^{\oslash n_0}})$, whose cycle has metric length $c_0$.
\end{lem}
\begin{proof} We write as in Example \ref{Ex:2.4.1} (c)  the graph $L$ as  $L=(V(L), E(L))$,
with  
 $$V(L)=\{x_i\}_{i=0}^k\cup\{ y_i^{(1)}\}_{i=0}^{l_1}\cup \{ y^{(2)}_i\}_{i=0}^{l_2}\cup \{z_i\}_{i=0}^m,$$
 where we make the following identifications:
 $$x_k\equiv y^{(1)}_0\equiv y^{(2)}_0,\text{ and }  y^{(1)}_{l_1}\equiv y^{(2)}_{l_2}\equiv z_0.$$
 \begin{align*} E(L)&=\big\{\{x_{i-1}, x_i\}: i=1,2,\ldots,k\big\} \cup \big\{\{z_{i-1}, z_i\}: i=1,2,\ldots,m\big\}\\
                               &\qquad   \cup  \big\{\{y^{(1)}_{i-1}, y^{(1)}_i\}: i=1,2,\ldots,l_1\big\}    \cup  \big\{\{y^{(2)}_{i-1}, y^{(2)}_i\}: i=1,2,\ldots,l_2\big\}.
                                  \end{align*}
We define the following four paths in $L$
\begin{align*}
P_1^{(1)}&=(x_i)_{i=0}^k\smile(y^{(1)}_i)_{i=0}^{l_1}\smile (z_i)_{i=1}^m\text{ and }
P_1^{(2)}=(x_i)_{i=0}^k\smile(y^{(2)}_i)_{i=0}^{l_2}\smile (z_i)_{i=1}^m,\\
Q_1^{(1)}&=(y^{(1)}_i)_{i=0}^{l_1}\text{ and }Q_1^{(2)}=(y^{(2)}_i)_{i=0}^{l_2}.
\end{align*}

For $n\in \N$, $n\ge 2$, we define inductively the following paths  $Q^{(1)}_n$, $Q^{(2)}_n$, which are subgraphs of $L^{\oslash n}$:
 $Q^{(1)}_n$ is obtained by replacing each  edge of $Q^{(1)}_{n-1}$, by $P_1^{(2)}$  
and $Q^{(2)}_n$ is  obtained by replacing each  edge of $Q^{(2)}_{n-1}$,  by the path $P_1^{(1)}$.
Note that $Q^{(1)}_n$ and $Q^{(2)}_n$ generate a cycle in  $L^{\oslash n}$ of length $c_0$.

Denote the graph lengths of $Q^{(1)}_n$ and $Q^{(2)}_n$ by $M_n$ and $N_n$. From  a simple induction argument it follows that
$$M_n =(k+m+l_2)^{n-1} l_1 \text{ and } N_n= (k+m+l_1)^{n-1} l_2.$$

\noindent and it is easy to see that for the above defined number $n_0$
$$ M_{n_0-1}\le N_{n_0-1}< N_{n_0}<M_{n_0}.$$
For $i=0,1,2,\ldots, M_{n_0-1}$, let   $Q^{(1)}_{n_0,i} $  be the  path which is 
obtained by replacing  $i$ edges of $Q^{(1)}_{n_0-1}$ by $P^{(2)}_1$ and 
$M_{n_0-1}-i$ edges by $P^{(1)}_1$.  
Note that $Q^{(1)}_{n_0,i} $ 
is a subgraph of $L^{\oslash n_0}$ which together  with $Q^{(2)}_{n_0}$  generates a cycle in $L^{\oslash n_0}$.
Let $M_{n_0,i}$ be the graph length of $Q^{(1)}_{n_0,i}$ and note that 
$$M_{n_0-1}<M_{n_0,i}= (k+m)M_{n_0-1} + i l_2 + (M_{n_0-1}-i)l_1\le M_{n_0}.$$

We will show that for some choice of $0\le i\le M_{n_0-1}$ 
 it follows that $M_{n_0,i}=N_{n_0}$.
Since 
\begin{align*}
M_{n_0,i}-N_{n_0}&=(k+m+l_1) (M_{n_0-1}-N_{n_0-1})+ i(l_2-l_1)\\
                              &=\begin{cases} (k+m+l_2) M_{n_0-1} -  (k+m+l_1) N_{n_0-1}                                                          &\text{ if $i=M_{n_0-1}$,}\\
                                                (k+m+l_1)(M_{n_0-1}-  N_{n_0-1}  )                                                            &\text{ if $i\keq0$, }\end{cases}\\
                                 &=\begin{cases} M_{n_0} - N_{n_0}   > 0                                                       &\text{ if $i=M_{n_0-1}$,  }\\
                                                (k+m+l_1)(M_{n_0-1}- N_{n_0-1}  )\le 0                                                            &\text{ if $i=0$,}\end{cases}                                            
                                                \end{align*}
                       
\noindent it is enough to show that $M_{n_0-1}-N_{n_0-1}$ is divisible by $l_2-l_1$, in order to deduce that $M_{n_0,i}=N_{n_0}$ for an appropriate choice of $i$.
To verify this we 
write 
\begin{align*}M_{n_0-1}-N_{n_0-1}&=(k+m+l_2)^{n_0-2} l_1-   (k+m+l_1)^{n_0-2} l_2\\
&=\sum_{j=0}^{n_0-2} (l_2^j\cdot l_1-l_1^j\cdot l_2) {n_0-2 \choose j}(k+m)^{n_0-2-j}.
\end{align*} 

For $j=0$,  $l_2^j\cdot l_1-l_1^j\cdot l_2= l_1-l_2$, and 
for $j\ge 1$ it follows that 
$l_2^j\cdot l_1-l_1^j\cdot l_2=l_1\cdot l_2 (l^{j-1}_2-l_1^{j-1})$
is also an integer multiple of $l_2-l_1$.
We conclude therefore that for some appropriate choice of $0\le i\le M_{n_0-1}$
the paths $Q^{(1)}_{n_0,i} $ and $Q^{(2)}_{n_0} $  have equal paths length  and generate a cycle  $C$ which is 
a subgraph of $L^{\oslash n_0}$. Finally let $R_1$ be  any path in $L^{\oslash n_0}$ from $s(L^{\oslash n_0})$ to the nearest point of $C$
and $R_2$ any path from the point of $C$, which is nearest to $t(L^{\oslash n_0})$, to  $t(L^{\oslash n_0})$, then 
$R_1$, $R_2$ and $C$ generate a balanced generalized Laakso graph which is a subgraph of $L^{\oslash n_0}$. The ``moreover'' portion of the statement is immediate from Lemma \ref{L:2.3.2}.
\end{proof} 

\begin{rem}\label{R:2.1}
    Lemma \ref{L:2.7.4} immediately implies that for all $n\geq n_0$, $(L^{\oslash n},d_{L^{\oslash n}})$ contains a balanced generalized Laakso $s$-$t$ subgraph. Indeed, if we have such a graph, say $L_1\subset G^{\varoslash n_1}$, for some $n_1\in \N$, then by the definition of slash powers the graph $L_2$ obtained from $L_1$ by replacing each edge with the same $s$-$t$ path from $G$, will be a balanced generalized Laakso subgraph  $L_2\subset G^{\varoslash (n_1+1)}$.
\end{rem}

\section{Embedding Cycles into Trees} \label{S:3}

The following is a generalization and sharpening of a result by Gupta \cite[Lemma 7.1]{Gupta2001}.  The proof, which we include for completeness, is similar.

\begin{lem}\label{L:3.1}
Let $(V(C),E(C),d_C)$ be a geodesic cycle, and let
$$c_0=\length_{d_C}(C)=\sum_{e\in E(C)} d_C(e)$$
be its length.

 Let $(T,d_T)$  be a geodesic tree.
Assume that there is an expansive 
map
$$\Psi: (V(C),d_C)\to (V(T), d_T)$$
(meaning that $d_C(u,v)\le d_T(\Psi(u),\Psi(v))$, for $u,v\in V(C)$). 

Then there exists an edge $e=\{u,v\}\in E(C)$, for which
$$d_T(\Psi(u), \Psi(v))\ge \frac{c_0-d_C(u,v)}8.$$

\end{lem} 
For the proof of Lemma \ref{L:3.1} we recall the following important result by Gupta.
\begin{thm}\label{T:3.2} {\rm \cite[Theorem 1.1]{Gupta2001}}  Let $T=(V(T),E(T),d_T)$  be a geodesic tree and let $V'\subset V$. Then there is a   tree $T'=(V(T'), E(T'), d_{T'})$ with $V(T')=V'$, and a geodesic metric $d_{T'}$ on $T'$, so that 
\begin{equation}\label{E:3.2.1} 1\le \frac{d_{T'}(x,y)}{d_{T}(x,y)}\le 8,\text{ for $x,y\in V'$}.\end{equation}
\end{thm}

\begin{proof}[Proof of Lemma \ref{L:3.1}] In view of Theorem \ref{T:3.2} it is enough to show the following claim: 

Let $(T,d_T)$ be a geodesic tree on $V(C)$, with the property that 
$$d_T(u,v)\ge d_C(u,v), \text{ for all $f=\{u,v\}\in E(T)$.}$$
Then there exists an edge $e=\{x,y\}\in E(C)$ for which
$$d_T(x,y)\ge c_0-d_C(x,y).$$

For each $f=\{u,v\}\in E(T)$ we can assume that 
\begin{equation}\label{E:3.2.2}
d_T(u,v)= d_C(u,v) = \sum_{j=1}^{m(f)} d_C ( x_{i-1}(f), x_i(f))\le c_0/2,
\end{equation} 
where  $(x_i(f))_{i=1}^{m(f)} $ is a shortest path from $u$ to $v$ in $C$ (which is unique if $d_C(u,v)<c_0/2$,  and if $d_C(u,v)=c_0/2$, there are two such paths).
Indeed, otherwise we could replace $d_T$ by the geodesic metric on $V(T)$ which 
is generated by the weight function $W: E(T)\to (0,c_0]$, $f=\{u,v\}\mapsto d_C(u,v)$.

For each $e=\{a,b\}\in E(C)$ we can write $d_T(a,b)$ as 
\begin{equation}
d_T(a,b)= \length_{d_T}([a,b]_T) = \sum_{j=1}^{n(e)}d_T(y_{j-1}(e),y_j(e))
\end{equation} 
where $[a,b]_T=(y_i(e))_{i=1}^{n(e)}$ denotes the (unique) path  from $a$ to $b$ in $T$.

After possibly passing to a tree $T'$ on $V(C)$ and a geodesic distance $d_{T'}$, 
 for which $d_{T'}(e)\le d_T(e)$, for all $e\in E(C)$,
we also can assume  that $d_T(e)<c_0/2$ for each $e\in E(T)$. This can be seen as follows:
Assume $d_C(a,b)=d_T(a,b)=c_0/2$ for some $e=\{a,b\}\in E(T)$.  Since $a$ and $b$ cannot be both leaves of $T$, we can assume that the degree of $b$ is at least $2$,
and thus that there exists a $c\in V(C)$, $a\not=c$, with $\{b,c\}\in E(T)$. It follows that $d_T(b,c)=d_C(b,c)<c_0/2$ (note that if $d_C(b,c)=d_C(a,b)=c_0/2$ it would follow that $a=c$).
We now consider the tree $T'$ on $V(C)$ whose edges  are 
$$E(T')=(E(T)\setminus\big\{\{a,b\}\big\})\cup \big\{\{a,c\}\big\}$$
(note that $\{a,c\}\not\in E(T)$, because otherwise $(a,b,c,a)$ would be  a cycle in $T$)
and the  geodesic distance $d_{T'}$, generated by the weight function
$W'(f)=d_{T}(f)$ if $f\in E(T)\setminus\big\{\{a,b\}\big\}$, and $W'(\{a,c\})= d_C(a,c)= d_C(a,b)- d_C(b,c)>0$.
It follows that for any $e=\{g,h\}\in E(C)$, $d_{T'}(e)=d_T(e)$, as long as 
$\{a,b\}$ is not an edge of the path $[g,h]_T$. If 
$\{a,b\}$ is an edge of the path $[g,h]_T$, we can replace in $[g,h]_T$ the edge $\{a,b\}$  by $\{a,c\}\smile\{c,b\}$ in order 
to get a walk in $T'$ from $g$ to $h$, whose length is at least as large as the length of path $[g,h]_{T'}$ in $T'$.

We suppose now that $V(C)=\{0,1,2,\ldots,n-1\}$ and $E(C)=\big\{\{i,i+1\}: i=0,1,2,\ldots,n-1\big\}$, with $n\ge3$, and where we  assume that addition in $V(C)$ is modulo $n$.
For $i\in V(C)$ we consider the clockwise  path in $C$ defined by
$S(i) =(i+j: j=1,2,\ldots, k_i)$ where 
$$k_i=\max \Big\{ k\in \N: \sum_{j=i+1}^{i+k} d_C(j-1,j)\le c_0/2\Big\}$$
and $S'(i)$ is the counter clockwise path consisting of the complement of $S(i)$, \ie\ 
$S'(i) = (i-j: j=1,2,\ldots , n-k_i-1)$.
We observe that for any $j\in S(i)$ or $j\in S'(i)$, there is a shortest path  from $i$ to $j$, which is   in $S(i)$, respectively $S'(i)$
(in the  limit case that $d_C(i,j)=c_0/2$, it follows from our convention that this path is in $S(i)$), and, thus,
it follows for $j\ge 0$ with  $i+j\in S(i)$, that 
\begin{equation}\label{E:3.2.3}
d_C(i,i+j)= \sum_{s=1}^{j} d_C(i+s-1, i+s).
\end{equation}

We also observe that for $u,v\in V(C)$, with 
$d_C(u,v)<c_0/2$, it follows that $u\in S(v)$, if and only if $v\in S'(u)$.

The following claim is crucial.

\noindent{\bf Claim.} After possibly passing to another tree  $T'$ on $V(C)$ and a geodesic distance $d_{T'}$, with 
$d_{T'}(e)\le d_T(e)$, for $e\in E(C)$, we may assume  that if  $\{v_0,u\}, \{v_0,w\}\in E(T)$, $u\not= w$, then 
 $u$ and $w$ cannot be  either  both in $S(v_0)$ or both in $S'(v_0)$.\\

\begin{figure}[h]
\centering
\begin{tikzpicture} 
    
    \draw[fill=none](0,0) circle (2.0) ;

    \draw[fill=black](-1.73205080757,-1) circle (1.5 pt) node [below left] {$v_0$};
    \draw[fill=black](-1.41421356237,1.41421356237) circle (1.5 pt) node [above left] {$u$};
    \draw[fill=black](0,2) circle (1.5 pt) node [above] {$w$};
    
    \draw[](-1.73205080757,-1) -- (-1.41421356237,1.41421356237) ;
    \draw[](-1.73205080757,-1) -- (0,2);
    \draw[dashed] (-1.41421356237,1.41421356237) -- (0,2);
        
    \end{tikzpicture}   
    \caption{We may assume $\{v_0,w\}\notin E(T)$, else, obtaining an improved tree by replacing the edge $\{v_0,w\}$  with $\{u,w\}$.}
\end{figure}
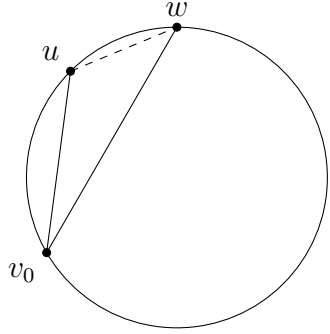

Indeed, assume, for example, that $u=v_0+s$, $w=v_0+t\in S(v_0)$, $t,s\ge 1$, and  that $d_C(v_0,u)<d_C(v_0,w)$, which by \eqref{E:3.2.3} means   that $0<s<t$.
Then we consider the tree $T'$ on $V(C)$ defined by $E(T')=( E(T)\setminus \big\{ \{v_0,w\}\big\})\cup \big\{\{u,w\}\big\}$,
and let $d_{T'}$ be the geodesic metric generated by the weight function
$W'(e)=d_T(e)$ if $e\in E(T)\setminus \big\{ \{v_0,u\}\big\}$ and 
$$W'(\{u,w\})= d_T(v_0,w)- d_T(v_0,u)= d_C(v_0,w)- d_C(v_0,u)<d_C(v_0,w) = d_T(v_0,w).$$
Then it follows for any $e=\{a,b\}\in E(C)$ that either  $[a,b]_T$ does not contain the edge $\{v_0,w\}$,
 in which case $d_T(a,b)=d_{T'}(a,b)$ or  $[a,b]_T$ does  contain the edge $\{v_0,w\}$. If we replace in $[a,b]_{T'}$ the edge
 $\{v_0,w\}$ by   $\{v_0,u\}$ and $\{u,w\}$
 we obtain a walk in $T'$, whose length with respect to $d_{T'}$
 equals to $d_T(a,b)$. Thus, proving the claim $d_{T'}(a,b)\le d_T(a,b)$.

The claim implies, in particular, that we can assume that the degree of each $v\in V(C)$ with respect to $T$ is at most $2$, and thus
that $T$ is a path, and we can order $V(C)$ into $x_0,x_1,\ldots x_{n-1}$ with $E(T)=\big\{ \{x_{i-1}, x_i\}: i=1,2,\ldots, n-1\big\}$.
After relabeling the elements of $V(C)$, we can assume that $x_0=0$, and that $x_1\in S(x_0)$. But this implies that $x_2\in S(x_1)$,
because otherwise $x_2\in S'(x_1) $ and $x_0\in S'(x_1)$ (here we use the assumption that $d_C(x_1,x_2)=d_T(x_1,x_2)<c_0/2$, and $d_C(x_0,x_1)=d_T(x_0,x_1)<c_0/2$), which is a contradiction to the above proven claim. Iterating this argument 
we conclude that $x_i\in S(x_{i-1})$, for $i=1,2,\ldots,n-1$.  Let $k\le n-1$ so that $x_k=n-1$. Then
\begin{align*} d_T(0,n-1)= \sum_{j=1}^k d_T(x_{j-1},x_j)= \sum_{j=1}^k d_C(x_{j-1},x_j)= c_0-d_C(0,n-1)
\end{align*}
which finishes the proof of our Lemma.
\end{proof}

\section{Embedding of Slash Powers of  Balanced Laakso Graphs into Trees}\label{S:4}

Throughout this section, $L$ is a fixed  balanced Laakso graph, and thus for some $k,l,m\in\N_0$, $l\ge 2$ we have
\begin{align*}
V(L)&=\{x_i\}_{i=0}^k\cup\{ y^{(1)}\}_{i=0}^{l}\cup \{ y^{(2)}_i\}_{i=0}^{l}\cup \{z_i\}_{i=0}^m,
  \text{ with }x_k\!\equiv\!y^{(1)}_0\!\equiv\!y^{(2)}_0, y^{(1)}_{l}\!\equiv\!y^{(2)}_{l}\equiv z_0\\
  \intertext{and }
E(L)&=\big\{\{x_{i-1}, x_i\}: i=1,2,\ldots,k\big\} \cup \big\{\{z_{i-1}, z_i\}: i=1,2,\ldots,m\big\}\\
                               &\qquad   \cup  \big\{\{y^{(1)}_{i-1}, y^{(1)}_i\}: i=1,2,\ldots,l\big\}    \cup  \big\{\{y^{(2)}_{i-1}, y^{(2)}_i\}: i=1,2,\ldots,l\big\}.
                                  \end{align*}
Let $d_L$ be a normalized geodesic metric on $V(L)$ and let 
$C_0$ be the cycle in $L$ generated by the paths $(y^{(1)}_i)_{i=0}^l$ and  $(y^{(2)}_i)_{i=0}^l$.
Note that 
$$c_0:=\length_{d_C}(C_0)= 2\sum_{j=1}^l d_L(y^{(1)}_{j-1}, y^{(1)}_j)= 2\sum_{j=1}^l d_L(y^{(2)}_{j-1}, y^{(2)}_j).$$

We require that 
\begin{equation} \label{E:4.1.1}
d_L(y^{(\sigma)}_{j-1},y^{(\sigma)}_j)\le \frac{c_0}4\le\frac12, \text{ for $j=1,2,\ldots l$, $\sigma=1,2$.}
\end{equation}

Let $\nu_L$ be the probability on $E(L)$ defined in Example \ref{Ex:2.4.1} (c).

For $n\in \N$ we abbreviate the $n$-th slash power of the measured normalized geodesic $s$-$t$- graph $(L,d_L,\nu_L)$ by $(L_n, d_n,\nu_n)$.

We write each $e\in E(L_n)$ as $e=e_1\oslash e_2\oslash\ldots e_n$, with $e_1,e_2,\ldots,e_n\in E(L)$  and each $v\in V(L_n)$
as $v=e_1\oslash e_2\oslash\ldots e_{n-1} \oslash u$ with $e_1,e_2,\ldots,e_{n-1}\in E(L)$ and $u\in V(L)$. For $e=e_1\oslash e_2\oslash\ldots\oslash e_n\in E(L_n)$ we put $e|_{[2,n]}=e_2\oslash e_3\ldots\oslash e_n\in E(L_{n-1})$
and $e|_{[1,n-1]}=e_1\oslash e_2\ldots\oslash e_{n-1}$.

For $n\in\N$ we put
$$\cC^{(n)}=\big\{ C=(V(C),E(C)): C \text{ is a cycle in $L_n$ of metric length $c_0$}\big\}, $$
which are the cycles of largest metric length in $L_n$,
and for  $e\in E(L_n)$ let
$$\cC^{(n)}_e=\big\{ C\in\cC^{(n)}: e\in E(C)\big\}.$$

We deduce from Lemma \ref{L:2.3.2},  Lemma \ref{L:3.1}  and \eqref{E:4.1.1}  the following Corollary.

\begin{cor}\label{C:4.1} Let $(T,d_T)$ be a geodesic tree, $n\in\N$  and  $\Psi: V(L_n)\to V(T)$ be an expansive map. Then for every $C\in \cC^{(n)}$ there is an edge $e_C\in E(C)$ for which
 $$d_T(\Psi(e_C))\ge \frac{c_0-d_n(e_C)}8\ge  \frac{3c_0}{32}.$$
\end{cor}

\begin{prop}\label{P:4.1} Let $n\in\N$.
\begin{enumerate}
\item[a)] For every $C\in \cC^{(n)}$ it follows that $|E(C)|= 2 l (k+l+m)^{n-1}.$
\item[b)] $\big|\cC^{(n)}\big|= 2^{2l{\frac{(k+l+m)^{n-1}-1}{k+l+m -1}}}$.\\
 In particular, using the formula for geometric series, we deduce for $n\ge 2$,
that  $\big|\cC^{(n)}\big|= \big|\cC^{(n-1)}\big| 2^{2l(k+l+m)^{n-2}}$.

\item[c)] Let $n\ge 2$ and $e=e_1\oslash e_2\oslash\ldots\oslash e_{n-1}\oslash e_n\in E(L_n)$. If $e_1\in E(C_0)$, then 
$$\big|\cC^{(n)}_e|=\big|\cC^{(n-1)}_{e|_{[1,n-1]} }\big|\cdot \begin{cases} 2^{2l (k+l+m)^{n-2}-1} &\text{ if $e_n\in E(C_0)$,}\\
 2^{2l (k+l+m)^{n-2}} &\text{ if $e_n\in E(L)\setminus  E(C_0)$.}\end{cases}$$
 If $e_1\in E(L)\setminus E(C)$, then $\cC^{(n)}_e=\emptyset$.
\end{enumerate}
\end{prop}
\begin{proof} (a) For $n=1$ the claim is clear. Assuming the claim is true for some $n\ge 1$. Then the claim for $n+1$ follows from the fact
that one obtains an element of $\cC^{(n+1)}$ by taking an element $C$ from $\cC^{(n)}$ and replacing each edge of $C$ by a path from $s(L)$ to $t(L)$  in $L$, whose
graph length is  $k+l+m$.

\noindent(b) Again for $n=1$, the claim is clear. Assuming the claim is true for $n$, then  the fact that an element of $C^{(n+1)}$ is obtained
by starting with a cycle  $C$ in $\cC^{(n)}$ and replacing each edge of  $C$ either by the  path $(x_i)_{i=0}^k\smile (y^{(1)}_i)_{i=1}^l\smile(z_i)_{i=0}^m$
or the  path $(x_i)_{i=0}^k\smile (y^{(2)}_i)_{i=0}^l\smile(z_i)_{i=0}^m$. This means that for each 
$C$ in $\cC^{(n)}$ there are $2^{|E(C)|}$ possibilities to extend $C$ to an element of $\cC^{(n+1)}$. Thus by
(a) and the induction hypothesis
$$\big|\cC^{(n+1)}\big |= \big|\cC^{(n)}\big| 2^{ 2 l (k+l+m)^{n-1}} = 2^{2l{\frac{(k+l+m)^{n-1}-1}{k+l+m -1}}+2 l (k+l+m)^{n-1} }=2^{2l\frac{(k+l+m)^{n}-1}{k+l+m -1}}.$$

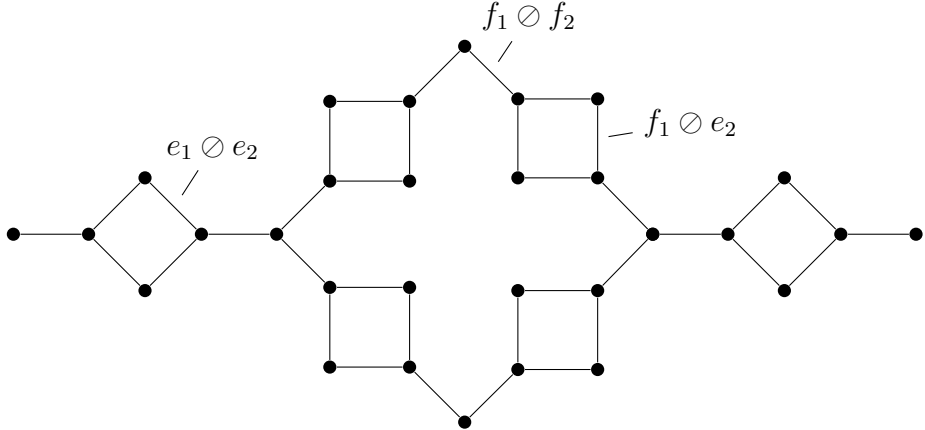
\begin{figure}[h]
\centering

\begin{tikzpicture}[roundnode/.style={circle, fill=black, inner sep=0pt, minimum size=1.75mm}]

\node (a) at (2.65,1.15) {$e_1\oslash e_2$};
\node (b) at (2.15,.375) {};

\node (c) at (6.85,2.9) {$f_1\oslash f_2$};
\node (d) at (6.35,2.15) {};

\node (e) at (9,1.475) {$f_1\oslash e_2$};
\node (f) at (7.7675,1.275) {};

\draw[-] (a) to (b);
\draw[-] (c) to (d);
\draw[-] (e) to (f);

\node[roundnode] (x0) at (0,0) {};
\node[roundnode] (x1) at (1,0) {};
\node[roundnode] (x2) at (1.75,.75) {};
\node[roundnode] (x3) at (1.75,-.75) {};
\node[roundnode] (x4) at (2.5,0) {};
\node[roundnode] (x5) at (3.5,0) {};

\node[roundnode] (w1) at (4.207,.707) {};
\node[roundnode] (w2) at (5.2675,.707) {};
\node[roundnode] (w3) at (4.207,1.7675) {};
\node[roundnode] (w4) at (5.2675,1.7675) {};
\node[roundnode] (w5) at (6,2.5) {};

\node[roundnode] (y1) at (4.207,-.707) {};
\node[roundnode] (y2) at (5.2675,-.707) {};
\node[roundnode] (y3) at (4.207,-1.7675) {};
\node[roundnode] (y4) at (5.2675,-1.7675) {};
\node[roundnode] (y5) at (6,-2.5) {};

\node[roundnode] (a1) at (6.707,1.8) {};
\node[roundnode] (a2) at (7.7675,1.8) {};
\node[roundnode] (a3) at (6.707,.75) {};
\node[roundnode] (a4) at (7.7675,.75) {};
\node[roundnode] (a5) at (8.5,0) {};

\node[roundnode] (b1) at (6.707,-1.8) {};
\node[roundnode] (b2) at (7.7675,-1.8) {};
\node[roundnode] (b3) at (6.707,-.75) {};
\node[roundnode] (b4) at (7.7675,-.75) {};
\node[roundnode] (b5) at (8.5,0) {};

\node[roundnode] (z1) at (9.5,0) {};
\node[roundnode] (z2) at (10.25,.75) {};
\node[roundnode] (z3) at (10.25,-.75) {};
\node[roundnode] (z4) at (11,0) {};
\node[roundnode] (z5) at (12,0) {};

\draw[-] (x0) -- (x1);
\draw[-] (x1) -- (x2);
\draw[-] (x1) -- (x3);
\draw[-] (x3) -- (x4);
\draw[-] (x2) -- (x4);
\draw[-] (x4) -- (x5);
\draw[-] (x5) -- (w1);
\draw[-] (x5) -- (y1);

\draw[-] (w1) -- (w2);
\draw[-] (w1) -- (w3);
\draw[-] (w2) -- (w4);
\draw[-] (w3) -- (w4);
\draw[-] (w4) -- (w5);
\draw[-] (w5) -- (a1);

\draw[-] (y1) -- (y2);
\draw[-] (y1) -- (y3);
\draw[-] (y2) -- (y4);
\draw[-] (y3) -- (y4);
\draw[-] (y4) -- (y5);
\draw[-] (y5) -- (b1);

\draw[-] (a1) -- (a2);
\draw[-] (a1) -- (a3);
\draw[-] (a2) -- (a4);
\draw[-] (a3) -- (a4);
\draw[-] (a4) -- (a5);

\draw[-] (b1) -- (b2);
\draw[-] (b1) -- (b3);
\draw[-] (b2) -- (b4);
\draw[-] (b3) -- (b4);
\draw[-] (b4) -- (b5);
\draw[-] (b5) -- (z1);

\draw[-] (z1) -- (z2);
\draw[-] (z1) -- (z3);
\draw[-] (z2) -- (z4);
\draw[-] (z3) -- (z4);
\draw[-] (z4) -- (z5);

\end{tikzpicture}

        \caption{The second slash power of a $(1,2,2,1)$-Laakso Graph. Note that for the three labeled edges $\left|\mathcal{C}^{(2)}_{e_1\oslash e_2}\right|=0$, $\left|\mathcal{C}^{(2)}_{f_1\oslash f_2}\right|=16=2^4$, and $\left|\mathcal{C}^{(2)}_{f_1\oslash e_2}\right|=8=2^3$.}
        \end{figure}

 \noindent(c) If $n=2$ (see Figure 3 above), and $(e_1\oslash e_2)\in E( L_2)$ with  $e_1\in E(C_0)$, then a cycle $C\in \cC^{(2)} $, contains $e_1\oslash e_2$ if it is obtained 
from   $C_0$ (which is the only element of $\cC^{(1)}$) by replacing each edge by the path   $(x_i)_{i=0}^k\smile (y^{(1)}_i)_{i=1}^l\smile(z_i)_{i=0}^m$
or the  path $(x_i)_{i=0}^k\smile (y^{(2)}_i)_{i=1}^l\smile(z_i)_{i=0}^m$. If $e_2$ is an edge of the path $(x_i)_{i=0}^k$ or $(z_i)_{i=0}^m$ this can be done in 
$2^{|E(C_0)|}=2^{2l}$ ways. But in the case that $e_2$ is an edge of either  the path  $(y^{(1)}_i)_{i=1}^l$ or  $(y^{(2)}_i)_{i=1}^l$, then 
there are only $2^{|E(C_0)|-1}=2^{2l-1}$ ways to do such. If $e_1\in E(L)\setminus E(C)$, then $e_1\oslash e_2$ cannot be an edge of any $C\in \cC^{(2)}$.

Assuming now that our claim is true for some $n\ge 2$ and $e=e_1\oslash e_2\oslash \ldots\oslash  e_{n+1}\in\E( L_{n+1})$, we can proceed similarly 
to verify the claim for $n+1$. Indeed, we obtain the  elements $C$ of $\cC^{(n+1)}_{e}$, by replacing  each edge of an element 
$C'\in \cC^{(n)}_{e|_{[1,e_{n}]}}$ by  the  path $(x_i)_{i=0}^k\smile (y^{(1)}_i)_{i=0}^l\smile(z_i)_{i=0}^m$
or the  path $(x_i)_{i=0}^k\smile (y^{(2)}_i)_{i=0}^l\smile(z_i)_{i=0}^m$. Again, if $e_n$ is an edge of the path $(x_i)_{i=0}^k$ or $(z_i)_{i=0}^m$ this can be done in 
$2^{|E(C')|}=2^{2l  (k+l+m)^{n-1}}$ ways. But in the case that $e_2$ is an edge of either  the path  $(y^{(1)}_i)_{i=0}^l$ or  $(y^{(2)}_i)_{i=0}^l$, then 
there are only $2^{|E(C')|-1}=2^{2l (k+l+m)^{n-1}-1}$ ways to do such. As in the case $n=2$, it follows that if $e_1\in E(L)\setminus E(C)$, then $\cC_e^{(n+1)}=\emptyset$
since $\cC^{(n)}_{e|_{[1,n]}}=\emptyset$, by our induction hypothesis.
\end{proof}

\begin{cor}\label{C:4.2} 
Let $n\in\N$ and assume that $\phi$ is a map from $\cC^{(n)}$ to $E(L_n)$, with the  property that  $\phi_1(C)\in C_0$, where we write $\phi(C)$ as 
$\phi(C)=\phi_1(C)\oslash\phi_2(C)\oslash\ldots\oslash\phi_n(C)$, with $\phi_j(C)\in E(L)$, $j=1,2,\ldots,n$. 

Then 
\begin{align}\label{E:4.1.2}
 \sum_{C\in \cC^{(n)}} \frac1{|\cC^{(n)}_{\phi{(C)}}|}\frac{\nu_n(\phi(C))}{d_n(\phi(C))}=\frac12     .                                                                                                                                                              \end{align} 
\end{cor}
\begin{proof} From the inductive definition of $\nu_n(e)$ and $d_n(e)$ we deduce that
\begin{align*}
  \sum_{C\in \cC^{(n)}} \frac1{|\cC^{(n)}_{\phi(C)}|}\frac{\nu_n(\phi(C))}{d_n(\phi(C))}
&= \sum_{C\in \cC^{(n)} }\frac1{|\cC^{(n)}_{\phi(C)}|} \frac{\nu_{n-1}( \phi(C)|_{[1,n-1]})}{d_{n-1} ( \phi(C)|_{[1,n-1])})} \frac{\nu_1(\phi_n(C))}{d_1(\phi_n(C))}
\intertext{which, by Proposition \ref{P:4.1}, and the fact that $\frac{\nu_1(\phi_n(C))}{d_1(\phi_n(C))}=\frac12\iff \phi_n(C)\in E(C)$ and 
$\frac{\nu_1(\phi_n(C))}{d_1(\phi_n(C))}=1\iff \phi_n(C)\in E(L)\setminus E(C)$, in the case of $n\ge 2$, equals to}
&= \sum_{C\in \cC^{(n)} } \frac1{|\cC^{(n-1)}_{\phi(C)|[1,n-1]}|}\frac{\nu_{n-1}( \phi(C)|_{[1,n-1]})}{d_{n-1} ( \phi(C)|_{[1,n-1])})}2^{-2l (k+l+m)^{n-2}}
\intertext{Repeating  the same argument as in the second equation this equals to }
\sum_{C\in \cC^{(n)} }&\frac1{|\cC^{(n-2)}_{\phi(C)|[1,n-2]}|} \frac{\nu_{n-2}( \phi(C)|_{[1,n-2]})}{d_{n-2} ( \phi(C)|_{[1,n-2])})}   2^{-2l(k+l+m)^{n-3}-2l (k+l+m)^{n-2}}.\\
\intertext{Iterating this argument we finally obtain}
\sum_{C\in \cC^{(n)} }&\frac1{|\cC^{(1)}_{\phi_1(C)}| } \frac{\nu_1(\phi_1(C))}{d_1(\phi_1(C))}2^{-2l  \sum_{j=1}^{n-2}(k+l+m)^j}\\
  &=  \frac12 \sum_{C\in \cC^{(n)} } 2^{-2l\frac{(k+l+m)^{n-1}-1}{k+l+m-1}}=\frac12  \text{ (by Proposition \ref{P:4.1} (b))}
  \end{align*}
which verifies our statement.
\end{proof}

The following observation follows directly from the definition of the metric $d_{H\oslash G}$ for two geodesic $s$-$t$ graphs 
$H$ and $G$.
\begin{prop}\label{P:4.2}  For each $e_1\in E(L)$, and $n\ge 2$  the map
$$\Phi_{e_1}: V( L_{n-1})\to  V( L_{n}), \quad e_2\oslash e_3\oslash \ldots \oslash e_{n-1}\oslash v\mapsto e_1\oslash e_2\oslash e_3\oslash \ldots \oslash e_{n-1}\oslash v$$
is  graph isomorphism
 and 
 $$d_n\big(\Phi_{e_1}(u),\Phi_{e_1}(v)\big)= d_1(e_1) d_{n-1}(u,v) \text{ for $u,v\in V( L_{n-1})$}.$$
\end{prop}

We are now ready to state the main result of this section.
\begin{thm}\label{T:4.1}
Let $n\in\N$,  $(T, d_T)$ a geodesic tree and  $\Psi: V(L_n)\to V(T)$ be  an expansive map.

Then 
$$\D_{\nu_n}(\Psi)= \E_{\nu_n}\Big(\frac{d_T(\Psi(e))}{d_n(e)}\Big)=\sum_{e=\{x,y\}\in E(L_n)} \frac{d_T(\Psi(e))}{d_n(e)} \nu_n(e)\ge \frac{3}{128}  c_0n.$$
\end{thm}

\begin{proof}
We consider the map
$$F_n: E(L_n)\to \R, \qquad e\mapsto \min\Big(\frac{d_T(\Psi(e))}{d_n(e)}, \frac{3}{32} \frac{c_0}{d_n(e)} \Big).$$
We will show by induction for all $n\in\N$ that 
\begin{align}\label{E:4.2}
\E_{\nu_n}(F_n) \ge \frac{3}{128}  c_0n.
\end{align}

For $n=1$, \eqref{E:4.2} is true since it follows from \eqref{E:4.1.1}   that $\frac3{32}\frac{ c_0}{d_1(e)}\ge\frac{6}{32} c_0 $.  Assume that our claim is true for $n-1$,  where $n\ge 2$.

Considering for each  $e\in E(L_n)$ the cases $d_T(\Psi(e))\ge \frac{3}{16}c_0$ and  $d_T(\Psi(e))<  \frac3{16}c_0$, we obtain
\begin{align*}
\E_{\nu_n} (F_n)&\ge \frac3{32} c_0\sum_{\substack{e\in E(L_n)\\ d_T(\Psi(e))\ge \frac3{32} c_0 }}  \frac{\nu_n(e)}{d_n(e)}
+\sum_{\substack{e\in E(L_n),\\ d_T(\Psi(e))< \frac3{32} c_0 }} 
 \min\Big( d_T(\Psi(e)),\frac3{32} c_0 \Big) \frac{\nu_n(e)}{d_n(e)}\\ 
&\ge \underbrace{ c_0 \frac12 \frac3{32}  \sum_{\substack{e\in E(L_n),\\ d_T(\Psi(e))\ge \frac3{32} c_0 }} \frac{\nu_n(e)}{d_n(e)} }_{=A}+
 \underbrace{\sum_{e\in E(L_n)}  \min\Big( d_T(\Psi(e)),\frac3{64} c_0 \Big) \frac{\nu_n(e)}{d_n(e)}}_{=B}.
\end{align*}

 For each $C\in \cC^{(n)} $ choose, according to Corollary \ref{C:4.1}, an  edge $e^C=e^C_1\oslash e^C_2\oslash \ldots \oslash e^C_n\in E(C)$ for which  $d_T(\Psi(e_C))\ge \frac{3}{32}c_0 $. The value of $A$ can be estimated as follows:
Since for each $C\in \cC^{(n)}$, the edge  $e_C$ is an element of at most $|C^{(n)}_{e_C}|$  circles of $\cC^{(n)}$ it follows from Proposition \ref{P:4.1} (c) and Corollary \ref{C:4.2}, that
\begin{align} A\ge c_0\frac3{64} \sum_{e\in \{e_C: C\in \cC^{(n)} \}}\frac{\nu_n(e)}{d_n(e)}  \ge  c_0\frac3{64} \sum_{C\in \cC^{(n)}} \frac{\nu_n(e_C)}{d_n(e_C)} \frac1{|C^{(n)}_{e_C}|}= \frac3{128}c_0.
\end{align} 
In order to estimate $B$ we define $P_1=(x_i)_{i=0}^k\smile (y^{(1)}_i)_{i=0}^l\smile(z_i)_{i=0}^m$ and $P_2=(x_i)_{i=0}^k\smile (y^{(2)}_i)_{i=0}^l\smile(z_i)_{i=0}^m$
(\ie\ the two paths from  $s(L)$ to  $t(L)$ in $L$)
and compute 
\begin{align*} B&=\sum_{e_1\in E(L)} \sum_{e'\in E( L_{n-1})} \min\Big(d_T(\Psi(e_1\oslash e')), \frac3{64} c_0 \Big) \frac{\nu_{n-1}(e')}{d_{n-1}(e')} \frac{\nu_1(e_1)}{d_1(e_1)}\\
&= \frac12\sum_{e_1\in E(P_1)} \sum_{e'\in E( L_{n-1})} \min\Big(d_T(\Psi(e_1\oslash e')),  \frac3{64} c_0\Big) \frac{\nu_{n-1}(e')}{d_{n-1}(e')}\\ &\qquad+
 \frac12 \sum_{e_1\in E(P_2)} \sum_{e'\in E( L_{n-1})} \min\Big(d_T(\Psi(e_1\oslash e')),  \frac3{64} c_0\Big) \frac{\nu_{n-1}(e')}{d_{n-1}(e')}.
\end{align*}
The above equality  is true since  $ \frac{\nu_1(e_1)}{d_1(e_1)}=1\iff e_1\in E( (x_i)_{i=0}^k)$ or  $e_1\in E( (z_i)_{i=0}^m)$ $\iff e_1\in E(P_1)\cap E(P_2)$ and 
  $ \frac{\nu_1(e_1)}{d_1(e_1)}=\frac12 \iff e_1\in E( (y^{(1)}_i)_{i=0}^l)$ or  $e_1\in E( (y^{(2)}_i)_{i=0}^l)$.
 For $\sigma=1,2$ we compute
  \begin{align*} 
\sum_{e_1\in E(P_\sigma)} &\sum_{e'\in E( L_{n-1})} \min\Big(d_T(\Psi(e_1\oslash e')), \frac3{64} c_0\Big) \frac{\nu_{n-1}(e')}{d_{n-1}(e')} \\
&= \sum_{e_1\in E(P_\sigma) } d_1(e_1) \sum_{e'\in E( L_{n-1})}  \min\Big(\frac{d_T(\Psi(e_1\oslash e')}{d_1(e_1)} ,\frac1{d_1(e_1)}  \frac3{64}c_0 \Big)  \frac{\nu_{n-1}(e')}{d_{n-1}(e')} \\
&\ge \sum_{e_1\in E(P_\sigma) } d_1(e_1)\sum_{e'\in E( L_{n-1})}  \min\Big(\frac{d_T(\Psi(e_1\oslash e'))}{d_1(e_1)} , \frac3{32}c_0\Big)\frac{\nu_{n-1}(e')}{d_{n-1}(e')}\quad \text{ (by \eqref{E:4.1.1})}\\
&=\sum_{e_1\in E(P_\sigma) }  d_1(e_1) \sum_{e'\in E(  L_{n-1})}\min\Big( d'_{T} (\Psi_{e_1}(e')), \frac3{32} c_0\Big) \frac{\nu_{n-1}(e')}{d_{n-1}(e')},
 \end{align*}  
 where for $e_1\in E(P_\sigma)$  we define $\Psi_{e_1} : V( L_{n-1})\to V(T) $, $v\mapsto \Psi (e_1\oslash v)$, which by Proposition \ref{P:4.2} is an expansive map 
 for the metric $d'_T(\cdot, \cdot)= \frac{d_T(\cdot, \cdot)}{d_L(e_1)}$ on $V(T)$.
 
 It follows, therefore, from the induction hypothesis that for each $\sigma=1,2$
 \begin{align*} 
\sum_{e_1\in E(P_\sigma)} \sum_{e'\in E( L_{n-1})} &\min\Big(d_T(e_1\oslash e'), c_0  \frac3{64} \Big) \frac{\nu_{n-1}(e')}{d_{n-1}(e')} \ge   \frac3{128} c_0 (n-1).
\end{align*} 
and, thus, that  $B\ge   \frac3{128} c_0 (n-1) $ which yields
$$\E_{\nu_n} (F_n)\ge A+B\ge   \frac3{128} c_0 n,$$
and finishes the induction step, and, thus, the proof of the claim.
\end{proof}
\section{Proof of the Main Theorem}

The following result is an immediate consequence of  Lemma \ref{L:2.3.2}, Proposition  \ref{P:2.7.3}, and Theorem \ref{T:4.1}.
\begin{cor}\label{C:5.1}
    Let $(G,d_G)$ be a normalized, geodesic $s$-$t$ graph such that $G$ contains a cycle, and let    $$c_0=\max\big\{\length_{d_G}(C)\ :\ C \text{ is cycle in } G\big\}.$$

    Suppose  $N\in \N$  is such that $G^{\oslash N}$ contains a balanced generalized Laakso $s$-$t$ subgraph, say $L$, whose cycle has length  $c_0$, and  such that $d_{G^{\oslash N}}(e)\leq \frac{c_0}{4}$ for all $e\in E(L)$ and $\nu$ is defined as in Example \ref{Ex:2.4.1} (c) on $E(L)$ and vanishes on $E(G^{\oslash N})\setminus E(L)$.  Let $n\in\N$,  $(T, d_T)$ be  a geodesic tree, and $\Psi: V(G^{\oslash Nn})\to V(T)$ be  an expansive map.

Then 
$$\D_{\nu_n}(\Psi)= \E_{\nu_n}\Big(\frac{d_T(\Psi(e))}{d_{Nn}(e)}\Big)=\sum_{e=\{x,y\}\in E(L^{\oslash n})} \frac{d_T(\Psi(e))}{d_{Nn}(e)} \nu_n(e)\ge \frac{3}{128}  c_0n.$$
\end{cor}

\begin{proof}[Proof of Main Theorem]
The result is clear in the case that $G$ is a tree. Indeed, if $G$ is a tree and an $s$-$t$ graph it
must be a path, and, thus, all slash powers are paths and the result follows trivially. Hence, we may suppose that $G$ contains a cycle. Let $C_0$ be a cycle in $G$ with metric length $c_0$. By Corollary \ref{C:5.1} and the observation that for any $k\in \N$, $|V(G^{\oslash k})|$ is of the order $|V(G)|^k$ it is sufficient to show that for a large enough $N\in \N$, $G^{\oslash N}$ contains a balanced generalized Laakso $s$-$t$ subgraph, $L$, where $d_{G^{\oslash N}}(e)\leq \frac{c_0}{4}$ for all $e\in E(L)$. However, this follows from the following observations that are consequences of the previous results of this paper:

\begin{enumerate}

    \item We can find a generalized Laakso $s$-$t$ subgraph, say $L_1$, of $G^{\oslash 2}$ such that for all $e\in E(L_1)$, $d_{G^{\oslash 2}}(e)<1$. Indeed, we note that since $G$ has a cycle it must contain an $s$-$t$ path with a graph length of at least 2, say $P$. Then in $G^{\oslash 2}$ there must be a cycle $C_1\subset G^{\oslash 2}$ formed by replacing each edge of $C_0$ with the path $P$. By Lemma $\ref{L:2.3.2}$ (1) we can find a generalized Laakso $s$-$t$ subgraph which by the proof of Lemma $\ref{L:2.3.2}$ (1) must satisfy $d_{G^{\oslash 2}}(e)<1$. Note too that $d_{G^{\oslash 2}}(C_1)=c_0$.

    \item There is a large enough $N_1\in \N$ so that for all $n\geq N_1$, $d_{L_1^{\oslash n}}(e)<\frac{c_0}{4}<\frac{1}{2}$ for any $e\in E(L_1^{\oslash n})$. Indeed,  if $\delta:=\max\{d_{G^{\oslash 2}}(e)\ :\ e\in E(L_1)\}$, then by the definition of the metric given on slash powers the first such $n$ such that $\delta^n<\frac{c_0}{4}$ will do.

    \item By the Remark following Lemma \ref{L:2.7.4}, there is an $N_2\in \N$ large enough so that for all $n\geq N_2$, $(L_1^{\oslash n},d_{L_1^{\oslash n}})$ contains a balanced generalized Laakso subgraph. It follows by the proof of Lemma \ref{L:2.7.4} and the same remark that the cycle in this balanced Laakso graph will have metric length $c_0$.

\end{enumerate}

Hence, there exists an $N_3\in \N$ so that, $(L_1^{\oslash N_3},d_{G^{\oslash 2N_3}})$ contains the desired balanced generalized Laakso $s$-$t$ subgraph.
\end{proof}
%
%
%
%
%
%
%
%
%
%
%
%
%





\begin{bibsection}\begin{biblist}

\bib{AlonKarpPelegWest1995}{article}{
  author={Alon, N.},
  author={Karp, R.},
  author={Peleg, D.},
  author={West, D.},
  year={1995},
  month={02},
  pages={78-100},
  title={A Graph-Theoretic Game and Its Application to the k -Server Problem},
  volume={24},
  journal={SIAM J. Comput.},
}

\bib{Bartal1996}{article}{
  author={Bartal, Y.},
  title={Probabilistic approximation of metric spaces and its algorithmic applications},
  booktitle={37th {A}nnual {S}ymposium on {F}oundations of {C}omputer {S}cience ({B}urlington, {VT}, 1996)},
  pages={184--193},
  publisher={IEEE Comput. Soc. Press, Los Alamitos, CA},
  year={1996},
  mrclass={68Q25 (90C27)},
  mrnumber={1450616},
}


\bib{BhattLeighton1984}{article}{
  author={Bhatt, S.},
  author={Leighton, T.},
  title={A framework for solving VLSI graph layout problems},
  journal={Journal of Computer and System Sciences},
  volume={28},
  number={2},
  pages={300-343},
  year={1984},
}


\bib{BMSZ2022}{article}{
    AUTHOR = {Baudier, F.},
    author={ Motakis, P.},
    author={Schlumprecht, T.},
     authour={Zs\'{a}k, A.},
     TITLE = {Stochastic approximation of lamplighter metrics},
   JOURNAL = {Bull. Lond. Math. Soc.},
  FJOURNAL = {Bulletin of the London Mathematical Society},
    VOLUME = {54},
      YEAR = {2022},
    NUMBER = {5},
     PAGES = {1804--1826},
      ISSN = {0024-6093},
   MRCLASS = {46B85 (05C05 05C12)},
  MRNUMBER = {4512686},
}

\bib{DilworthKutzarovaOstovskii2021}{article}{
author = {Dilworth, S.},
author = {Kutzarova, D.},
author = {Ostrovskii, M.},
year = {2021},
month = {09},
pages = {1-33},
title = {Analysis on Laakso graphs with application to the structure of transportation cost spaces},
volume = {25},
journal = {Positivity},
doi = {10.1007/s11117-021-00821-w}
}

\bib{FakcharoenpholRaoTalwar2004}{article}{
  author={Fakcharoenphol, J.},
  author={Rao, S.},
  author={Talwar, K.},
  title={A tight bound on approximating arbitrary metrics by tree metrics},
  journal={J. Comput. System Sci.},
  fjournal={Journal of Computer and System Sciences},
  volume={69},
  year={2004},
  number={3},
  pages={485--497},
  issn={0022-0000},
  mrclass={05C10 (68R10 68W25)},
  mrnumber={2087946},
  mrreviewer={A. Vijayakumar},
}

\bib{Gupta2001}{article}{
  author={Gupta, A.},
  title={Steiner points in tree metrics don't (really) help},
  booktitle={Proceedings of the {T}welfth {A}nnual {ACM}-{SIAM} {S}ymposium on {D}iscrete {A}lgorithms ({W}ashington, {DC}, 2001)},
  pages={220--227},
  publisher={SIAM, Philadelphia, PA},
  year={2001},
  mrclass={05C85 (51E10)},
  mrnumber={1958411},
}

\bib{GuptaNewmanRabinovichSinclair1999}{article}{
  author={Gupta, A. },
  author={ Newman, I.},
  author={Rabinovich, Y.},
  author={Sinclair, A.},
  year={1999},
  month={02},
  pages={399 - 408},
  title={Cuts, trees and l1-embeddings of graphs},
  isbn={0-7695-0409-4},
  journal={Foundations of Computer Science, 1975., 16th Annual Symposium on},
}

\bib{KleinRaoAgrawalRavi1995}{article}{
   author={Agrawal, A.},  
   author={Rao, S.},
  author={Ravi, R.},
  title={An approximate max-flow min-cut relation for undirected multicommodity flow, with applications},
  journal={Combinatorica},
  year={1995},
  volume={15},
  pages={187-202},
}



\bib{LeeRaghavendra2010}{article}{
    AUTHOR = {Lee, J.},
    author={Raghavendra, P.},
     TITLE = {Coarse differentiation and multi-flows in planar graphs},
   JOURNAL = {Discrete Comput. Geom.},
  FJOURNAL = {Discrete \& Computational Geometry. An International Journal
              of Mathematics and Computer Science},
    VOLUME = {43},
      YEAR = {2010},
    NUMBER = {2},
     PAGES = {346--362},
      ISSN = {0179-5376},
   MRCLASS = {05C10 (05C85 68R10)},
  MRNUMBER = {2579701},
}




\bib{LeightonRao1999}{article}{
  author={Leighton,T.},
  author={Rao, S.},
  title={Multicommodity Max-Flow Min-Cut Theorems and Their Use in Designing Approximation Algorithms},
  year={1999},
  publisher={Association for Computing Machinery},
  volume={46},
  number={6},
 pages={787-832},
}


\bib{Mathey-PrevotValette2021}{article}{
  author={Mathey-Prevot, M.},
  author={Valette, A.},
  title={Wasserstein distance and metric trees},
  journal={preprint, arXiv2110.02115v1},
}

\bib{RabinovichRaz1998}{article}{
  author={Rabinovich, Y. },
  author={Raz, R.},
  title={Lower bounds on the distortion of embedding finite metric spaces in graphs},
  journal={Discrete Comput. Geom.},
  fjournal={Discrete \& Computational Geometry. An International Journal of Mathematics and Computer Science},
  volume={19},
  year={1998},
  number={1},
  pages={79--94},
}


 \end{biblist}

\end{bibsection}

\end{document}